\documentclass[12pt]{article}
\usepackage{latexsym,amssymb,amsmath,amsfonts,amsthm}
\usepackage{graphicx}
\usepackage{tikz}
\usepackage{mathrsfs}
\numberwithin{equation}{section}
\newtheorem{theorem}{Theorem}[section]
\newtheorem{lemma}[theorem]{Lemma}
\newtheorem{prop}[theorem]{Proposition}
\newtheorem{corollary}[theorem]{Corollary}
\newtheorem{conj}[theorem]{Conjecture}

\def\beq{ \begin{equation} }
\def\eeq{ \end{equation} }
\def\bal{ \begin{align*}}
\def\eal{ \end{align*}}
\def\mn{\medskip\noindent}

\def\ep{\epsilon}

\def\square{\vcenter{\vbox{\hrule height .4pt
  \hbox{\vrule width .4pt height 5pt \kern 5pt
        \vrule width .4pt} \hrule height .4pt}}}

\def\RR{\mathbb{R}}
\def\TT{\mathcal{T}}
\def\ZZ{\mathbb{Z}}

\def\NN{\mathcal{N}}
\def\VV{\mathbb{V}}
\def\WW{\textbf{W}}
\def\BB{\textbf{B}}
\def\XX{\textbf{X}}

\def\KK{\mathcal{K}}

\def\clearp{}

\def\hh{\hspace{1ex}}

\def\varep{\varepsilon}
\def\norm#1{\langle #1\rangle}

\begin{document}

\title{Motion by mean curvature \\ in interacting particle systems}
\author{Xiangying Huang and Rick Durrett }

\date{\today}						

\maketitle

\centerline{\it To Harry Kesten, prodigious problem solver, mentor, colleague, and friend}

\begin{abstract} 
There are a number of situations in which rescaled interacting particle systems have been shown to converge to a reaction diffusion equation (RDE) with a bistable reaction term, see e.g., \cite{CDP13,Wald,evog,DN94}.  These RDEs have traveling wave solutions. When the speed of the wave is nonzero, block constructions have been used to prove the existence or nonexistence of nontrivial stationary distributions. Here, we follow the approach in a paper by Etheridge, Freeman, and Pennington \cite{Etheridge} to show that in a wide variety of examples when the RDE limit has a bistable reaction term and traveling waves have speed 0, one can run time faster and further rescale space to obtain convergence to motion by mean curvature. This opens up the possibility of proving that the sexual reproduction model with fast stirring has a discontinuous phase transition, and that in Region 2 of the phase diagram for the nonlinear voter model studied by Molofsky et al \cite{Molofsky} there were two nontrivial stationary distributions. 
\end{abstract}

\section{Introduction}

The literature on motion by mean curvature is extensive, so we will only cite the papers most relevant to our research. In 1992 Evans, Soner, and Souganidis  \cite{ESS} established that suitably rescaled versions of the Allen-Cahn equations converge to motion by mean curvature, in the sense that the solution converges to an indicator function of a region whose boundary evolves as the mean curvature flow. The big breakthrough made in this paper was that the limiting result was valid for all time despite the possible occurrence of geometric singularities. See the first four pages of \cite{ESS} for the physical motivation and references to previous work.

In 1995 Katsoulakis and Souganidis \cite{KS} used the results developed in \cite{ESS} to prove that stochastic Ising models with long range interactions, called Kac potentials, when rescaled converge to motion by mean curvature. The interaction kernel for their Ising model on $\ZZ^d$ is
\beq
\begin{cases}
K_\gamma(x,y) = \gamma^d J(\gamma |x-y|)\\
\text{$J : \RR^d \to [0,\infty)$ has compact support and is symmetric, i.e., $J (x) = J(|x|)$.}
\label{assJ}
\end{cases}
\eeq
The weighted sum of spins seen by $x$ is
$$
h_\gamma(x) = \sum_{y \neq x} K_\gamma(x,y) \sigma(y).
$$
This formula  is used to define the Gibbs measure with inverse temperature $\beta$
$$
\mu(\sigma) = \frac{1}{Z(\beta)} \exp\left( - \beta\sum_x h_\gamma(x) \sigma(x)\right),
$$
where $Z(\beta)$ is a normalization to make $\mu$ a probability measure. For this formula to be meaningful we have to restrict to a finite box $\Lambda = [-L,L]^d$ with boundary conditions imposed outside of $\Lambda$ and then let $L \to\infty$. See Chapter 6 of Liggett \cite{L85} for more details. $h_\gamma$ is also used 
to define the rates at which $\sigma(x)$ flips to $-\sigma(x)$,
$$
c_\gamma(x,\sigma) = \frac{ \exp(-\beta h_\gamma(x) \sigma(x)) }
{  \exp(-\beta h_\gamma(x)) +  \exp(\beta h_\gamma(x)).}
$$
This is one in the large collection of flip rates for which Gibbs states are reversible stationary distributions. Again,  see Chapter 6 of \cite{L85}. 

A very basic question is to understand the behavior of the process as $\gamma\to 0$. DeMasi, Orlandi, Presutti, and Trioli \cite{DOPT1,DOPT2,DOPT3} studied the limits as $\gamma\to0$ of the averaged magnetization of the system
$$
m_\gamma(x,t) = E^\gamma_{\mu^\gamma} \sigma_t(x), \quad (x,t)\in \ZZ^d\times\RR^+,
$$
where $ E^\gamma_{\mu^\gamma}$ is the expectation starting from the measure $\mu^\gamma$.
To state the result in \cite{DOPT1} we need the mean-field equation
\beq
\frac{\partial m}{\partial t} +m - \tanh(J \ast m)=0 \quad \text{in }\RR^d\times \RR^+,
\label{MFeq}
\eeq
where $J\ast m$ denotes the usual convolution in $\RR^d$. Let $\ZZ^d_n=\{ \bar{x}=(x_1,\dots,x_n)\in (\ZZ^d)^n |x_1\neq\cdots\neq x_n\}$.

\begin{theorem}[Theorem 2.1 in \cite{KS}] 
Assume that the initial measure is product measure $\mu^\gamma$ with
$$
E^\gamma_{\mu^\gamma} \sigma(x) = m_0(\gamma x), \quad x\in \ZZ^d,
$$
where $m_0$ is Lipschitz continuous and \eqref{assJ} holds. Then for any fixed $n$
and $\bar{x}\in \ZZ^d_n$,
$$
\lim_{\gamma\to 0} \bigg| E^\gamma_{\mu^\gamma} \left( \prod_{i=1}^n \sigma_t(x_i) \right)- \prod_{i=1}^n m(\gamma x_i,t)\bigg|=0
$$
where $m$ is the unique solution of \eqref{MFeq} with initial condition $m_0$.
\end{theorem}

\noindent
In words, the distribution of the particle system at time $t$ is almost a product measure in which the probabilities are given by $m(\gamma x,t)$.  To prove convergence to motion by mean curvature \cite{KS} use a lengthy argument to examine the asymptotics of the mean-field equation \eqref{MFeq} as $t\to\infty$ and space and time are rescaled.
Since the publication of \cite{KS} a number of similar results have been proved. \cite{Bona,FS95,FT19,Sowers,Yip} is a small sample of the papers that can be found in AMS subject classification 60.

\subsection{A more probabilistic approach}

Soon after the publication of \cite{ESS}, Chen \cite{Chen92} generalized much of this work and simplified the proofs. Etheridge et al \cite{Etheridge} use his paper as their primary source of information about motion by mean curvature, so we will as well. The object of study in \cite{Chen92} is the reaction diffusion equation (RDE)
\beq\label{AllenCahn}
\begin{cases}
\frac{\partial u}{\partial t} = \Delta u - \frac{1}{\ep^2}f(u),\quad &(x,t)\in \RR^d\times \RR^+,\\
u(x,0)=p(x), \quad & x\in \RR^d,
\end{cases}
\eeq
where $\ep$ is a small rescaling parameter, $p$ is a bounded continuous function in $\RR^d$ and $f$ is the derivative of a bistable potential. 
Chen gives general conditions on $f$ in (1.3) of his paper \cite{Chen92} that guarantee motion by mean curvature will appear in the limit as $\ep\to 0$, 
\begin{align*}
f \in C^2(\RR), \quad & \text{$f$ has exactly three zeros: $u_-< u_0< u_+$}\\
f(u)<0 , \quad & \text{for $u\in (-\infty, u_-)\cup (u_0,u_+)$}\\
f(u)>0, \quad & \text{for $ u\in (u_-,u_0)\cup (u_+,\infty)$}\\
f'(u_-)>0, \quad &f'(u_+)>0, \quad f'(u_0)<0.
\end{align*}
We will restrict our attention to the case in which $f$ is a third or fifth degree polynomial that is anti-symmetric around its central root $u_0$, i.e., $f(u_0-x) = - f(u_0+x)$. 

In the case of a cubic, the $1/\ep^2$ in front of the reaction term suggests that when $\ep$ is small the values of the solution will be close to one of the three fixed points ($u_-$, $u_0$ and $u_+$) across most of the space.  Chen's results prove this and give quantitative estimates when $\ep$ is small. 

To explain the phrase ``motion by mean curvature'', we note that under some assumptions that we state later, he proved that the set of points $\{ x\in \RR^d: u(x,t)=u_0\}$ can be written as a family of parameterized hyper-surfaces $\Gamma_t: S^{d-1} \to \RR^d$ where $S^{d-1}$ is the unit sphere in $\RR^d$, and $\Gamma_t$ evolves by 
\beq\label{mcflow}
\frac{\partial \Gamma_t(\theta)}{\partial t} = \kappa_t(\theta) n_t(\theta), \quad \theta\in S^{d-1},
\eeq
where $n_t(\theta)$ is the vector normal to the hypersurface and $\kappa_t(\theta)$ is the mean curvature, i.e., the sum of the principal curvatures. We refer to $\mathbf{\Gamma}=\{ \Gamma_t: t\geq 0\}$ as the mean curvature flow.

Etheridge et al \cite{Etheridge} used Chen's results to show that  the spatial $\Lambda$-Fleming-Viot process with selection against heterozygosity when suitably rescaled in space and time converges to motion by mean curvature. We refer the reader to \cite{Etheridge} for the description of the process. Their first step was to study the behavior of the PDE in $d\geq 2$,
$$
\frac{\partial v^\ep}{\partial t} 
= \Delta v^\ep + \frac{1}{\ep^2} v^{\ep}(1-v^\ep)(2v^{\ep}-1),
\qquad v^\ep(0,x) = p(x)
$$
where $p(x):\RR^d \to [0,1]$ is the initial condition. To analyze the PDE \cite{Etheridge} introduce a branching Brownian motion in which particles split into 3 at a fixed rate $\ep^{-2}$. As in the systems described in the next subsection, this is a dual process that can be used to compute solutions of the PDE. To find $u(x,t)$ one starts with a particle at $x$ at time $t$ and runs the branching Brownian motion  down to time 0. If a particle in the system ends up at $y$ at time 0, its state is set to be $1$ with probability $p(y)$ and  $0$ with probability $1-p(y)$. As we work upwards the branching tree, states of particles do not change until three lineages coalesce into one. At this point the one lineage that emerges after coalescence takes the value that is in the majority of the three coalescing particles.

\begin{figure}[ht]
\begin{center}
\begin{picture}(260,260)
\put(130,230){\line(0,-1){200}}
\put(130,230){\line(-1,-2){60}}
\put(130,230){\line(2,-3){52}}
\put(70,110){\line(0,-1){80}}
\put(70,110){\line(-1,-3){27}}
\put(70,110){\line(1,-3){27}}
\put(182,150){\line(0,-1){120}}
\put(182,150){\line(-1,-4){30}}
\put(182,150){\line(1,-4){30}}
\put(182,90){\line(0,-1){60}}
\put(182,90){\line(-1,-4){15}}
\put(182,90){\line(1,-4){15}}
\put(20,30){\line(1,0){220}}
\put(40,15){1}
\put(67,15){0}
\put(95,15){1}
\put(125,15){1}
\put(65,120){1}
\put(150,15){1}
\put(165,15){0}
\put(180,15){0}
\put(195,15){1}
\put(210,15){0}
\put(185,90){0}
\put(182,155){0}
\put(130,235){1}
\end{picture}
\caption{Picture of the branching Brownian motion. We run from $(x,t)$ down to time 0, and then work back up the structure to compute the state of $x$ at time $t$.}
\end{center}
\end{figure}
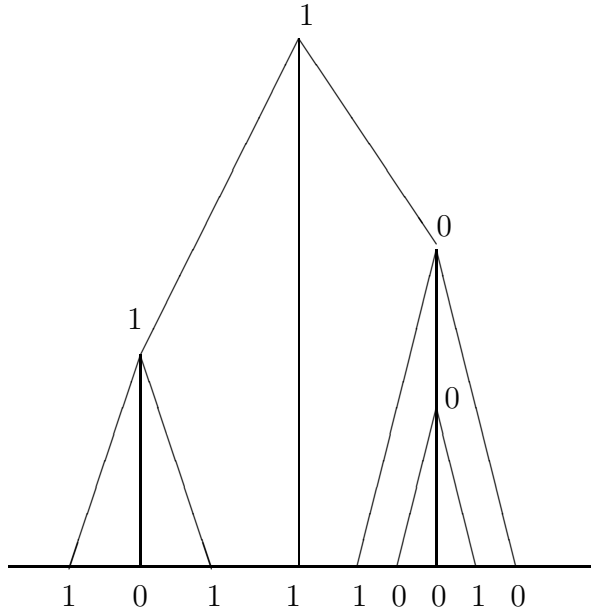

A variety of particle systems have dual processes that are close to branching Brownian motions after rescaling. A similar treatment as in \cite{Etheridge} can thus be taken to understand these systems. Instead of taking a majority vote at each branch point in the dual process, the specific interaction rule of the particle system considered would prescribe the value of the lineage that emerges after that. 

In what follows we will discuss three examples. The sexual reproduction model which is a system with fast stirring and the Lotka-Volterra system and nonlinear voter models that are examples of voter model perturbations. In each case we will first consider a system $\xi^\delta_t$ on $\delta \ZZ^d$ run at rate $\delta^{-2}$ that converges to a reaction diffusion equation. Then we will introduce a process $\xi^\ep_t$ on $\delta\ep\ZZ^d$ that is further sped up by a factor of $\ep^{-2}$ that converges to motion by mean curvature. For reasons that we will explain later we will choose 
\beq
\delta=\exp(-\ep^{-3})\quad\hbox{or}\quad
\ep=(\log(1/\delta))^{-1/3}.
\eeq
Note that $\ep$ is determined by $\delta$ and vice versa so we can regard either as the small parameter in $\xi^\ep_t$, but given the notation for the processes it is more natural to choose $\ep$.

\subsection{Systems with fast stirring}

Particle systems with fast stirring were first introduced by Durrett and Neuhauser \cite{DN94}. Let $\delta>0$ be a small rescaling parameter. They considered processes $\xi^{\delta}_t:\delta \ZZ^d\to \{0,1\}$ that evolve as follows:

\mn
(i) There are translation invariant finite range flip rates $c_\delta(x,\xi)$ that give the rate at which site $x$ changes to the opposite state when the configuration is $\xi$. 

\mn
(ii) For each unordered pair $x,y\in \delta\ZZ^d$ with $\|x-y\|_1=\delta$ we exchange the values at $x$ and $y$ at rate $\delta^{-2}/2$.

\medskip

We will focus on the special case in which the particle system is the ``sexual reproduction'' model where state 1 means a site is occupied and state  0 means vacant. The flip rates is given by
$$
c_\delta (x,\xi) =   1_{\{\xi(x)=1\}} + 1_{\{\xi(x)=0\}} \cdot \lambda n_1(x,\xi), 
$$
where $\lambda>0$ is the birth rate and $n_1(x,\xi)$ is the number of pairs in the set 
$$
x+\NN^\delta_b\equiv x+\delta \cdot \big\{\{e_1,e_2\}, \{-e_1, e_2\}, \{-e_1,-e_2\}, \{e_1,-e_2\} \big\}
$$ 
in which both sites are in state 1. To have a concrete example in mind we will restrict our attention to the case $d=2$. As there are four possible pairs in $\ZZ^2$, we let $\beta=4\lambda$.

Durrett and Neuhauser \cite{DN94} showed that as $\delta\to 0$ the density of 1's near $x$ at time $t$ converges to the solution of
\beq
\frac{\partial u}{\partial t }= \frac{1}{2} \Delta u + \phi(u) \quad\hbox{where}\quad \phi(u) = -u + \beta u^2(1-u).
\label{SR}
\eeq
The term $-u$ in $\phi(u)$ accounts for deaths of individuals (i.e., the flips from 1 to 0 at rate 1), while the term $\beta u^2(1-u)$ accounts for the sexual reproduction. Writing $\phi(u) = - u ( 1 - \beta u(1-u) )$ we see that when $\beta < 4$ there is no positive solution of $\phi(u)=0$. When $\beta=4$, 1/2 is a double root.  When $\beta>4$ there are two positive roots $\rho_1 < 1/2 < \rho_2 < 1$.  Based on this calculation one might guess that as $\delta\to 0$, the critical value for survival of the sexual reproduction model with fast stirring,  $\beta_c(\delta)$, should converge to 4. However, the correct result, which is proved in \cite{DN94}, is $\beta_c(\delta) \to 4.5$ as $\delta \to 0$.

To explain the intuition behind this, we look at the PDE \eqref{SR} in $d=1$ for intuition. We note that if $\beta > 4$  there are traveling wave solutions $u(x,t) = w(x-ct)$ with $w(-\infty) = \rho_2$ and $w(\infty)=0$. A little calculus shows that $w$ satisfies
$$
-c w' = (1/2) w'' + \phi(w).
$$
Multiplying by $w'$ and integrating from $-\infty$ to $\infty$, we find, see (1.6) in \cite{DN94}, that
$$
c \int w'(x)^2 \, dx = \int_0^{\rho_2} \phi(y) \, dy.
$$
We have no idea about the value of  $\int w'(x)^2 \, dx$, but it is positive so this tells us that the sign of the wave speed $c$ is the same as that of the integral on the right-hand side. When $\beta = 4.5$, the three roots are $0$, $1/3$, and $2/3$, so symmetry around the central root $1/3$ implies the integral is 0. Monotonicity (or calculus) tells us that $c<0$ when $\beta<4.5$, and $c>0$ for $\beta > 4.5$. Convergence results for the PDE, see e.g. \cite{FM77}, and
block constructions were used to show that

\begin{itemize}
\item
When $\beta> 4.5$ there is some $\delta_0(\beta)>0$ such that for $\delta< \delta_0(\beta)$ 
 there is a nontrivial stationary distribution with a density close to $\rho_2$. The second part of the conclusion about density is an improvement due to Cox, Durrett, and Perkins \cite{CDP13}.

\item
When $\beta < 4.5$  there is some $\delta_0(\beta)>0$ such that for $\delta< \delta_0(\beta)$  the process $\xi^\delta_t$ dies out.
\end{itemize}

\mn
Since $\rho_2(\beta)$ approaches $2/3$ as $\beta\downarrow4.5$, it is conjectured in  \cite{DN94} the density of the upper invariant measure (which is obtained by starting with all 1's and letting $t\to\infty$) has a positive density at $\beta_c(\delta)$ when $\delta$ is small. 

Here we speed up the process $\xi^\delta_t$ by an extra factor $\ep^{-2}$ and rescale the space to $\delta\ep\ZZ^d$ to obtain a new process 
$$
\xi^{\ep}_t: \delta\ep \ZZ^2 \to\{0,1\}.
$$  
If $\ep$ is kept fixed then the limiting differential equation as $\delta\to 0$ is
\beq\label{resSR}
\frac{\partial u^\ep}{\partial t }= \frac{1}{2} \Delta u^\ep + \frac{1}{\ep^2} \phi(u^\ep), \quad u^\ep(0,x)=p(x),
\eeq
where $p: \RR^d \to [0,1]$ is the initial condition and the reaction term remains the same
$$ 
\phi(u) = -u + \beta u^2(1-u).
$$ 
\eqref{resSR} matches the form of an Allen-Cahn equation given in \eqref{AllenCahn} except for a factor $\frac{1}{2}$ in front of the Laplacian. This is because their underlying Brownian motions have different rates. The Brownian motion with generator $\Delta u$ has rate 2, that is, at time 1 the Brownian motion has variance is 2, while the Brownian motion with generator $\frac{1}{2}\Delta u$ runs at rate 1. We will adopt the convention in probability and assume that
\beq
\text{ all Brownian motions have rate 1,}
\eeq
which gives rise to PDEs with a factor $\frac{1}{2}$ before the Laplacian like \eqref{resSR}.

Fixing $\ep$ and letting $\delta\to 0$ shows us how the rescaled particle system is related to the Allen-Cahn equation. However, to prove our result we need to take both $\ep$ and $\delta$ to 0. In order to avoid collisions in the dual process (see Section \ref{dual} for a full discussion), we need to require that the branching rate $\ep^{-2}$ is much slower than the stirring rate $(\delta\ep)^{-2}/2$ so that newly born particles move away from each other before the next branching time. Choosing $\delta = \exp(-\ep^{-3})$ guarantees this. Weaker conditions may suffice.

Let $p(x):\RR^d\to [0,1]$ be the initial density of the system that we consider. In the case of sexual reproduction $d=2$. We will state our assumptions on $p$ later in Section \ref{proofintro}, see (C1)-(C3). We say the process $\xi^\ep_t$ starts with initial condition $p$ if the initial distribution is a product measure where $P(\xi^\ep_0(x)=1)=p(x)$ for $x\in \delta\ep\ZZ^d$.

\begin{theorem}\label{sexlim}
Let $\xi^\ep_t: \delta\ep\ZZ^2\to \{0,1\}$ denote the rescaled sexual reproduction model with fast stirring starting with an initial condition $p(x)$ that satisfies (C1)-(C3). Choose $\delta=\exp(-\ep^{-3})$. If $\beta=4.5$ then as $\ep\to 0$, $P(\xi^\ep_t(x)=1)$ converges to motion by mean curvature.
\end{theorem}

\noindent
Theorem \ref{main0} will explain explicitly what it means to converge to motion by mean curvature. Theorem \ref{sexlim} shows that the probabilities $P(\xi^\ep_t(x)=1)$ converge to a density $u(x,t)$ that satisfies motion by mean curvature.  As in Theorem 1.2 in \cite{CDP13} one can also prove that the rescaled particle system which takes values in $\{0,1\}$ on a fine grid also converges to $u(x,t)$. See the discussion before Theorem 1.2 in \cite{CDP13} for the necessary definition. This remark also applies to the next two examples.

In motion by mean curvature the interfaces become straight as time $t\to\infty$, so the regions in which the solution is close to one of the two stable fixed points get larger. This suggests that

\begin{conj} If $\beta=4.5$ there exists some $\ep_0(\beta)>0$ so that when $\ep< \ep_0(\beta)$ there is a translation invariant stationary distribution for the process $\xi^\ep_t$ with density close to 2/3.
\end{conj}

\noindent
Theorem \ref{sexlim} suggests that there is a discontinuous phase transition but does not rule out the possibility that the phase transition could be continuous for any $\ep>0$.

\subsection{Voter model perturbations}\label{introvmp}

Cox, Durrett and Perkins  \cite{CDP13} introduced  a class of interacting particle systems called voter model perturbations. For simplicity we will restrict our attention to processes with two states.  In this case the process is denoted by $\xi_t:\ZZ^d\to \{0,1\}$ and  the rate at which $\xi_t(x)$ flips to the opposite state given configuration $\xi$ is 
$$
c^\delta(x,\xi)=c_v(x,\xi)+\delta^2c_p(x,\xi)
$$
where $c_v(x,\xi)$ is the voter flip rate and $c_p(x,\xi)$ is the perturbation flip rate. We rescale the system $\xi_t$ by $x\to \delta x, t\to \delta^{-2} t$ and obtain the rescaled process $\xi^\delta_t: \delta\ZZ^d\to \{0,1\}$.  The perturbation $c_p(x,\xi)$ is scaled down by $\delta^2$ so that on the sped up time scale it is $O(1)$ while the voter model runs at rate $\delta^{-2}$.

The voter model part of the process will depend on a symmetric (i.e, $K(x) = K(-x)$), irreducible probability kernel $K : \ZZ^d \to [0, 1]$ with $K(0) = 0$ and covariance matrix $\sigma^2 I$. Letting $\NN_v$ denote the neighborhood for voting (determined by $K$), whenever there is a voter flip at $x\in \ZZ^d$, the voter at $x$ chooses a site in $x+\NN_v$ randomly according to the probability kernel $K$ and adopts its state. The voter flip rate can be formulated as 
$$
c_v(x,\xi)= (1-\xi(x))f_1(x,\xi)+\xi(x)f_0(x,\xi),
$$
where $f_j(x,\xi)=\sum_{y\in \ZZ^d}K(y-x)1_{\{\xi(y)=i\}}$ is the local density.

Cox, Durrett, and Perkins \cite{CDP13} have shown (see their Theorem 1.2) that, under some mild assumptions on the perturbation $c_p$, if we run the system on $\delta\ZZ^d$ with $d \ge 3$ then the process converges to the solution of a reaction diffusion equation
$$
\frac{\partial u}{\partial t} = \frac{\sigma^2}{2} \Delta u + \phi(u),
$$
where $\phi$ is the reaction term that depends on the particular perturbation. A general formula is given in Secction 1.1 of \cite{CDP13}. See (1.30)
Here, $d \ge 3$ is needed so that the voter model has a one parameter family of stationary distribution. Four examples were studied in \cite{CDP13}. Two fall within the scope of this investigation. 

\mn
{\bf Lotka-Volterra systems.} This model of the competition of two species were initially studied by Neuhauser and Pacala \cite{NeuPac}. For more recent references see \cite{CDP13}. In this case the perturbation rate is given by
$$
c_p(x,\xi) = \theta_0 f_1^2(1-\xi(x)) + \theta_1 f_0^2 \xi(x)
$$
where $\theta_0$ and $\theta_1$ are parameters in $\RR$. In words we pick two nearest neighbors of $x$ (with replacement, according to $K$) and flip if both of the neighbors are of the opposite type to $x$. Let $\{e_1,e_2\}$ be i.i.d. with law $K(\cdot)$  and let $\norm{\cdot}_u$ denote the expectation on the product space where $e_1,e_2$ and $\xi$ are independent and $\xi$ distributes as the voter equilibrium with density $u$. Then the limiting PDE has reaction term
$$\phi(u)=\theta_0 \norm{ (1-\xi(0))\xi(e_1)\xi(e_2)}_u-\theta_1 \norm{ \xi(0)(1-\xi(e_1))(1-\xi(e_2))}_u.$$
This term can be rewritten in the form
\beq\label{lotka}
\phi(u) = u(1-u)[\theta_0p_2 - \theta_1(p_2+p_3) +up_3(\theta_0+\theta_1)],
\eeq
where $p_2=p(0|e_1,e_2)$ is the probability that the rate 1 random walks with kernel $K$ starting from $e_1$ and $e_2$ coalesce but they avoid the one starting at 0, and $p_3=p(0|e_1|e_2)$ is the probability that the random walks starting from $0,e_1,e_2$ never coalesce.

In \cite{CDP13} the phase diagram is described. There are five regions $\{R_i, 1\leq i\leq 5\}$, see Figure 1.1. At the boundary between $R_4$ and $R_5$, $\theta_0=\theta_1=\theta>0$ so \eqref{lotka} simplifies to
$$
\phi(u) =\theta p_3  u(1-u)(2u-1).
$$
In this case the reaction diffusion equation is bistable and the speed of traveling waves is 0. Next we further rescale the system $\xi^\delta_t$ by $x\to \ep x$, $t\to \ep^{-2}t$ to get the second rescaled process $\xi^\ep_t$.
Following the same approach as our proof of Theorem \ref{sexlim}, we have

\begin{theorem}\label{LVlim}
Let $\xi^\ep_t: \delta\ep\ZZ^d\to \{0,1\}$ denote the rescaled voter model perturbations where the perturbation is a Lotka-Volterra system, starting with an initial condition $p(x)$ that satisfies (C1)-(C3). Choose $\delta=\exp(-\ep^{-3})$. In $d \ge 3$ as $\ep\to 0$, $P(\xi^\ep_t(x)=1)$  converges to motion by mean curvature.
\end{theorem}

In the Lotka-Volterra system the stable fixed points are at 0 and 1, so reasoning as we did for the sexual reproduction model with fast stirring:

\begin{conj} When $\ep$ is sufficiently small there is clustering in the process $\xi^\ep_t$, i.e., for any finite box $B$ the probability of seeing both types in the box tends to 0 as $t\to\infty$. 
\end{conj}

\mn
{\bf Nonlinear voter models.} Molofsky et al  \cite{Molofsky} used simulations and heuristic arguments to study a discrete time system with nearest neighbor interactions. We consider a continuous time version of the system with long range interactions. At times of the arrivals of a rate 1 Poisson process, a site $x$ chooses four points $x_1, \ldots x_4$ at random from $x + [-L,L]^d$. If there are exactly $k$ one‘s at the sites $x, x_1, \ldots x_4$ then $x$ becomes 1 with probability $a_k$ and 0 with probability $1-a_k$ where 
$$
a_0=0, \quad a_5=1, \quad a_1=1-a_4 \quad a_2=1-a_3.
$$  
This gives us a two-parameter family of models that are symmetric under interchange of 0 and 1. 

It is complicated to compute the reaction term $\phi_L(u)$ explicitly as the states of the chosen sites $x, x_1, \ldots x_4$ might not be independent. However, when the neighborhood $\NN_b=[-L,L]^d$ is chosen to be large then coalescence in the dual process is rare and the states of these sites become nearly independent. A little calculation, see (1.67) in \cite{CDP13}, shows that if they are independent then the reaction term is
\beq
\phi(u) = b_1 u(1-u)^4 + b_2 u^2(1-u)^3 -b_2 u^3(1-u)^2 - b_1(1-u)^4 u,
\label{NLVphi}
\eeq 
where $b_1 = 4a_1-a_4$ and $b_2=6a_2 - 4a_3$. For any $L>0$, $\phi_L(u)$ has the same form as that in \eqref{NLVphi} with coefficients $b_{1,L},b_{2,L}$ instead of $b_1,b_2$. If $L$ is large then the coefficients $b_{1,L},b_{2,L}$ are close to the coefficients $b_1,b_2$ in the independent case.

The reaction term $\phi(u)$ is a cubic in Region $1$ and $3$, but in Region $2$ and $4$ it is quintic. This leads to the following predictions about the behavior of the system.
\begin{itemize}
  \item In Region 1, the fixed point at 1/2 is attracting, so the system should exhibit coexistence.
  \item In Region 3, the fixed point at 1/2 is unstable, so when the process is sped up it should exhibit motion by mean curvature, and we expect clustering, i.e., for any finite box $[-N,N]^d$ the probability that all sites in this box have the same state tends to 1.
  \item In Region 2, 0 and 1 are unstable fixed points, so if the fixed points are $u^* < 1/2 < 1-u^*$, the values in $[0,u^*-\varep]$ and $[1-u^*+\varep,1]$ for any $\varep>0$ should rapidly disappear from the solution. When the process is sped up then the system exhibits motion by mean curvature, resulting in large regions with 1's at density $u^*$ separated by a thin boundary from large regions with density $1-u^*$.
  \item In Region 4, there is a traveling wave solution $w_1$ with $w_1(-\infty) =1$  and $w_1(\infty) = 1/2$ with speed $c_1$ and  a traveling wave solution $w_2$ with $w_2(-\infty) =1/2$ and $w_2(\infty) = 0$ with speed $c_2$. By symmetry $c_2=-c_1$. If $c_1<0$ (Case 4A), the PDE converges to 1/2 and there is coexistence. If $c_1>0$ (Case 4B) and $L$ is sufficiently large, then there is a traveling wave solution $w_0$ of the PDE in $d=1$ with $w_0(-\infty) =0$ and $w_0(\infty) = 1$ with speed 0 (see page 284 in \cite{FM81}). When the process is sped up then it should exhibit motion by mean curvature, and we expect clustering.
\end{itemize}

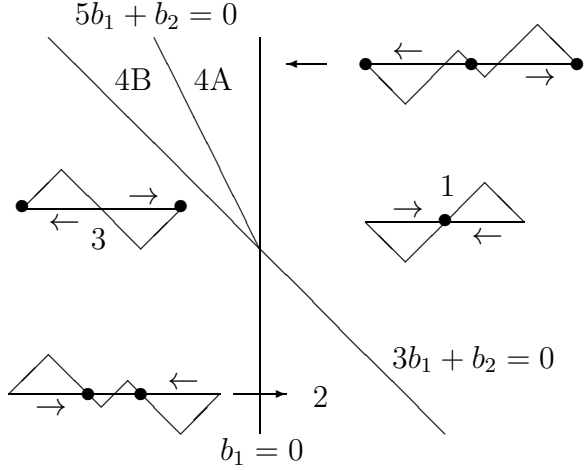
\begin{figure}[h]
\begin{center}
\begin{picture}(200,210)
\put(100,30){\line(0,1){150}}
\put(170,30){\line(-1,1){150}}
\put(100,100){\line(-1,2){40}}
\put(168,120){1}
\put(167,108){$\bullet$}
\put(140,110){\line(1,0){60}}
\put(140,110){\line(1,-1){15}}
\put(155,95){\line(1,1){30}}
\put(185,125){\line(1,-1){15}}
\put(180,102){$\leftarrow$}
\put(150,112){$\rightarrow$}
\put(36,100){3}
\put(67,113){$\bullet$}
\put(7,113){$\bullet$}
\put(10,115){\line(1,0){60}}
\put(10,115){\line(1,1){15}}
\put(25,130){\line(1,-1){30}}

\put(55,100){\line(1,1){15}}
\put(50,117){$\rightarrow$}
\put(20,107){$\leftarrow$}
\put(120,40){2}
\put(90,45){\vector(1,0){20}}
\put(5,45){\line(1,0){80}}
\put(5,45){\line(1,1){15}}
\put(20,60){\line(1,-1){20}}
\put(32,42){$\bullet$}
\put(40,40){\line(1,1){10}}
\put(50,50){\line(1,-1){20}}
\put(52,42){$\bullet$}
\put(70,30){\line(1,1){15}}
\put(15,37){$\rightarrow$}
\put(65,47){$\leftarrow$}
\put(45,160){4B}
\put(75,160){4A}
\put(125,170){\vector(-1,0){15}}
\put(137,167){$\bullet$}
\put(140,170){\line(1,0){80}}
\put(140,170){\line(1,-1){15}}
\put(155,155){\line(1,1){20}}
\put(175,175){\line(1,-1){10}}
\put(177,167){$\bullet$}
\put(185,165){\line(1,1){20}}
\put(205,185){\line(1,-1){15}}
\put(217,167){$\bullet$}
\put(150,172){$\leftarrow$}
\put(200,162){$\rightarrow$}
\put(85,20){$b_1=0$}
\put(30,185){$5b_1+b_2=0$}
\put(150,55){$3b_1+b_2=0$}
\end{picture}
\caption{Phase diagram for the continuous time nonlinear voter model with large range in $d\ge 3$. Piecewise linear curves show the shape of $\phi$. Black dots indicate the locations of stable fixed points }
\label{fig:nlvpt2}
\end{center}
\end{figure}

In \cite{CDP13} the following result is proved, see their Theorem 1.13.

\begin{theorem}
Suppose $(b_1,b_2)$ is in Region $1$, $2$ or $4A$. If $L$ is sufficiently large then
(i) There is coexistence for sufficiently small $\delta$ (depending on $L$). 
(ii) Let $\eta>0$. In Region 1 and 4A, there is a $\delta_0(\eta)$ so that for $\delta<\delta_0(\eta)$ and any stationary distributions $\nu$ with $\nu( \xi\equiv 0) = \nu(\xi\equiv 1)=0$ have 
$$
\sup_x \big|\nu(\xi(x)=1) - 1/2\big| < \eta.
$$
\end{theorem}

Again we need to further speed up the process to get convergence to motion by mean curvature. We rescale the process $\xi^\delta_t$ a second time by speeding up time by an extra $\ep^{-2}$ and rescaling space to $\delta\ep \ZZ^d$ to define a process $\xi^\ep_t$.

We define the voting and branching neighborhoods by
$$
{\cal N}_v = \{ \pm e_1, \ldots \pm e_d \} \quad \text{ and }\quad{\cal N}_b =[-L,L]^d \cap \ZZ^d.
$$
To prove our result we need several assumptions:

\medskip
\noindent
(A1) $b_1 >0$ and $3b_1+b_2 <0$: the process is in Region 2.\\
(A2) $0\le a_1 \le a_2 \le 1/2$: the process is attractive.\\
(A3) $6b_1+b_2>0$: the $g$ function defined in \eqref{g3} is concave on $(1/2,1-u^*)$.

\begin{theorem} \label{NLVlim}
Let $\xi^\ep_t: \delta\ep\ZZ^d\to \{0,1\}$ denote the rescaled voter model perturbations where the perturbation is a nonlinear voter model, starting with an initial condition $p(x)$ that satisfies (C1)-(C3). Choose $\delta=\exp(-\ep^{-3})$. Suppose the initial condition $p(x)$ satisfies $\varep \le p(x) \le 1-\varep$ for some $\varep>0$. In $d \ge 3$ if (A1), (A2), and (A3) hold then as $\ep\to 0$, $P(\xi^\ep_t(x)=1)$ converges to motion by mean curvature.
\end{theorem}

\noindent
Using the reasoning from the two previous examples:

\begin{conj} \label{cNLV}
If (A1), (A2), and (A3) hold then there exists some $\ep_0(a_1,a_2)>0$ so that when $\ep<\ep_0(a_1,a_2)$ there is a translation invariant stationary distribution in which the density is close to $u^*$. By symmetry there is also one with density to close to $1-u^*$.
\end{conj}

\noindent
The statement in Conjecture \ref{cNLV} implies the existence of two translation invariant stationary distributions.

\subsection{Overview of proofs}\label{overview}

The key to the proof in \cite{Etheridge} and in our three examples is understanding the dual process and using a special function $g$ to compute the state after each branch point in the dual process.

\mn
\textbf{Duality and the g function.} 

\mn
\textbf{The $g$ function in \cite{Etheridge}.}
\cite{Etheridge} begins by constructing a dual process that produces the solution to 
$$
\frac{\partial u}{\partial t} = \Delta u +c u(1-u)(2u-1), \qquad u(0,x) = p(x), x\in \RR^d.
$$ 
The initial condition $p$ is assumed to take values in $[0,1]$ and satisfy some regularity conditions that we will state later.

The dual process in \cite{Etheridge} is a branching Brownian motion in $\RR^d$ in which the Brownian motions are run at rate 2 and split into 3 particles at rate $c$. To compute the solution at $x$ at time $t$ they run the dual process backward in time down to time 0. A dual particle that lands at $y$ at time 0 is set to be $1$ with probability $p(y)$ and to be $0$ with probability $1-p(y)$. The states for different particles at time 0 are independent. Then they work their way back up the tree performing majority vote whenever three lineages merge into one. In \cite{Etheridge} an important role in the proof is played by the function
\beq
g_0(p) = p^3+ 3p^2(1-p) = 3p^2-2p^3,
\label{g0}
\eeq
which is the probability that the output of the majority vote operation is 1 when the inputs are independent Bernoulli($p$) random variables. $g_0$ has fixed points at $0$, $1/2$ and $1$.

\mn
\textbf{The $g$ function in sexual reproduction model with fast stirring.}
The sexual reproduction model with fast stirring has a dual process that was introduced by Durrett and Neuhauser \cite{DN94}. The dual has particles that are moved by stirring, and have births when events in the sexual reproduction dynamics occur. In Section 2 we define this dual process rigorously and show that in the fast stirring limit it is almost a branching Brownian motion in $\RR^2$. Since a birth event depends on the states of three particles (two particles in the chosen pair and the particle at the center), the dual branches into three particles at each branch point.
However, we mark one lineage to indicate it came from the original particle while the other two are offspring. When $\beta=4.5$, we have a birth event with probability $9/11$ (i.e., $\beta/(1+\beta)$) and a death event with probability 2/11. The analogous function to \eqref{g0} for the sexual reproduction model with fast stirring is
\beq
g_1(p) = \frac{9}{11} [p^2(1-p)  + p] = \frac{9}{11} [p + p^2- p^3],
\label{gsex}
\eeq
which has fixed points 0, $1/3$ and $2/3$.

\mn
\textbf{The $g$ function in voter model perturbations where the perturbation is a Lotka-Volterra system.}
 Voter model perturbations also have duals that were defined by Cox, Durrett, and Perkins \cite{CDP13}. In the class of nonlinear voter models that we will study, the dual is a branching coalescing random walk. In the Lotka-Volterra system the system branches into three, and we mark one lineage to indicate it came from the original particle (call it $x$) while the others are offspring (call them $y$ and $z$). 

To simplify computation, in the dual process we only consider the ``effective" perturbations and let them be branch points. When a perturbation occurs at $x$ there are three possibilities: (i) If $x$ coalesces with $y$ or $z$ (or with both) there is no change in the state of $x$. Hence we ignore this case. (ii) If $y$ and $z$ coalesce then $y$ and $z$ share the same state. This case is treated as a voter event since $x$ would adopt the state of a randomly chosen neighbor ($y$ or $z$). Hence this case is not part of the perturbation. (iii) If there is no coalescence among the three particles, $x$ changes state if $y$ and $z$ are both in the opposite state to itself. Case (iii) is the only effective perturbation and in this case 
\beq
g_2(p) = (1-p)p^2 + p[1-(1-p)^2] = 3p^2 - 2p^3
\label{gLV}
\eeq
which is the same as \eqref{g0}.

\mn
\textbf{The $g$ function in voter model perturbations where the perturbation is a nonlinear voter model.}
In the nonlinear voter model the system branches into five, and we again mark one lineage to indicate it came from the original particle. Since the branching rate is 1 $\phi(p)$ has the form 
$$\phi(p)=-p h_1(p)+(1-p)h_2(p),$$
where $h_1(p)$ represents the probability of getting a 0 when the center is in state 1 and the states of the rest four sites are i.i.d. Bernoulli($p$), while $h_2(p)$ represents the probability of the center flipping from 0 to 1. It follows that 
\begin{align}
\nonumber g_3(p) &= p(1-h_1(p))+(1-p)h_2(p)=\phi(p)+p \\
&= b_1p(1-p)^4 + b_2p^2(1-p)^3 - b_2p^3(1-p)^2 - b_1 p^4 (1-p) + p
\label{g3}
\end{align}
where $b_1=4a_1 - a_4$ and $b_2=6a_2-4a_3$. In the collection of nonlinear voter models that we study $g_3$ has fixed points at 0, $1-u^*$, 1/2, $u^*$, and 1. If the middle fixed point $u_0$ is unstable then 0,1 are stable if there are three zeros, and $1-u^*,u^*$ are stable if there are five zeros. We collect these observations into an assumption

\mn
(G0) There are fixed points $0 \le u_- < u_0 < u_+ \le 1$ where $u_0$ is unstable, $u_+, u_-$ are stable, and $u_+-u_0 = u_0 - u_-$. There can be fixed points at 0 and 1 which must be unstable. To avoid absorption, the initial condition $p(x)$ is uniformly bounded away from the fixed points at 0 and 1, i.e., \\
(i) if there is a fixed point at 0 we suppose the initial condition $p(x) \ge \varep$ for some $\varep>0$,\\
(ii) if there is a fixed point at 1 we suppose the initial condition $p(x) \le 1-\varep$ for some $\varep>0$.

\medskip
We can observe the reaction term $\phi(p)$ in the above three examples satisfies
\beq\label{gphi}
 \phi(p)=r(g(p)-p)
 \eeq
 where $r$ is the reaction rate.

\subsection{Proof of convergence}\label{proofintro}

The main result in this paper is the following result that contains Theorems \ref{sexlim}, \ref{LVlim} and \ref{NLVlim} as special cases. This result applies to any model with fast stirring or voter model perturbation where the $g$ function satisfies (G0) and the following assumptions:

\begin{align}
(G1) & \quad g(u_+-\delta)+g(u_-+\delta)=u_-+u_+=2u_0
\label{G1} \\
(G2) & \quad g'(u_0)>1\quad \text{ and }\quad g'(u_-)=g'(u_+)<1.
\label{G2} \\
(G3 )& \quad g''(p)>0 \text{ if }p\in (u_-,u_0),  g''(p)<0\text{ if }p\in (u_0,u_+).
\label{G3}
 \end{align}

\mn
(G4) There exists $c_0\in (0,1-g'(u_-))$ and $\delta_* \equiv \inf\{ x\geq 0 : g'(u_-+x) \ge 1-c_0\}>0$ so that for $\delta\leq \delta_*$ 
\beq\label{G4}
u_+-g(u_+-\delta)=g(u_-+\delta)-u_-\leq (1-c_0)\delta.
\eeq
(G5) $g$ is strictly increasing on $[0,1]$.

\mn
In Section \ref{sec:check} we will show that the conditions hold in all our examples. 

The initial condition $p: \RR^d\to [0,1]$ is assumed to satisfy some regularity conditions given later. Given $p$, the initial interface is defined to be 
$$
\Gamma=\bigg\{ x\in\RR^d: p(x)=u_0\bigg\}.
$$
Following \cite{Etheridge} we suppose that $\Gamma$ is a smooth hypersurface which is also the boundary of a bounded open set topologically equivalent to the sphere. Now we state the regularity conditions imposed on the initial condition $p$:

\mn
(C1) $\Gamma$ is $C^\alpha$ for some $\alpha>3$.

\mn
(C2) For $x$ inside $\Gamma$, $p(x)<u_0$. For $x$ outside $\Gamma$, $p(x)>u_0$.

\mn
(C3) There exists $r,\gamma>0$ such that for all $x\in\RR^d$, $|p(x)-u_0|\geq \gamma (dist(x,\Gamma)\wedge r).$

\mn
The conditions (C1)--(C3) guarantee that the mean curvature flow $\mathbf{\Gamma}=\{ \Gamma_t: t\geq 0\}$ started from $\Gamma$ exists up to some finite time $\mathscr{T}$, see e.g., Evans and Spruck \cite{ES91}. With $\Gamma_t$ properly defined, the meaning of $d(x,t)$ is now precise: it is the signed distance from $x$ to $\Gamma_t$, positive outside $\Gamma_t$ and negative inside. Note that $\Gamma_t=\{ x\in \RR^d: d(x,t)=0\}$.

In the scope of this paper we consider only the evolution of a single interface. In more general scenarios there could be multiple interfaces evolving together, say nested interfaces.

\begin{theorem}\label{main0}
Let $\xi^\ep_t$ denote a rescaled particle system on $\delta\ep\ZZ^d$ within the two classes considered. Suppose $\xi^\ep_t$ satisfies (G0)--(G5), and let $u^\ep(t,x)=P(\xi^\ep_t(x)=1)$ with $u^\ep(0,x)=p(x)$. Let $T^*\in(0,\mathscr{T})$ and $k\in \mathbb{N}$ be fixed. There exist $\ep_d(k)>0$ and $b_d(k),c_d(k)\in (0,\infty)$ such that for all $\ep\in(0,\ep_d)$ and $t$ satisfying $b_d(k)\ep^2|\log \ep|\leq t\leq T^*$,

\begin{enumerate}
\item for $x$ such that $d(x,t) \geq c_d(k)\ep|\log \ep|$, we have 
$u^{\ep}(t,x)\in (u_+-\ep^k, u_+ + \ep^k)$,
\item for $x$ such that $d(x,t)\leq -c_d(k)\ep|\log \ep|$, we have 
$u^{\ep}(t,x)\in (u_- -\ep^k,u_-+\ep^k)$.
\end{enumerate}
\end{theorem}

Our proof follows \cite{Etheridge} very closely. As we give our proof we will describe the correspondence between the two arguments. Here we give a brief outline of the proof to highlight the main steps. From now on we will let $\BB(t)$ denote the one-dimensional branching Brownian motion, while $\WW(t)$ denotes branching Brownian motion in $d\geq 2$.

\mn
{\bf Step 1.} Prove a result in one dimension. Let $u_-, u_+$ be the stable fixed points of $g$, and let $\VV(\BB(t))$ be the result of applying the algorithm defined in Section \ref{compp} to compute the state when the initial density is $p_0(x)=u_- \cdot 1_{\{x<0\}}+u_+\cdot 1_{\{x\geq 0\}}$. In $d=1$ the interface is a single point and there is no curvature so it does not move. The one dimensional version of Theorem \ref{main0}  is Theorem \ref{thm2.5}.   These results are proved by combining facts about the iteration,  with information on the structure of the tree and bounds on the movement of Brownian motion.

\mn
{\bf Step 2}. Generalize Theorem \ref{thm2.5} to $d\ge 2$ with $x$ replaced by $d(x,t)$, the signed distance from the $x$ to the interface $\Gamma_t$, see Proposition \ref{215}. 

\mn
{\bf Step 3.} Proposition \ref{215} takes care of the values away from the interface. The next step is to take care of the values near the interface by showing that the probability the dual gives a 1 (resp. 0) at $x$ for a general initial condition $p$ is almost the same as the probability the algorithm in Section \ref{compp} computes a 1 (resp. 0) at $d(x,t)\pm  K_1e^{K_2t}\ep|\log\ep|$ in the one dimensional system with the special initial condition $p_0(x)=u_- \cdot 1_{\{x<0\}}+u_+\cdot1_{\{x\geq 0\}}$. See Proposition \ref{2.16} for this result. The key step to proving Proposition \ref{2.16} is  Lemma \ref{217}, which compares the values computed by the algorithm in $d=1$ at 
\begin{align*}
z^\pm_1 & = d(\hat{X}^\ep_s,t-s)\pm \gamma(t-s)\ep|\log \ep|,\\
z^\pm_2 & = B_s\pm \gamma(t)\ep|\log \ep|,
\end{align*}
where $\hat{X}^\ep_s$ is an approximation of the rescaled dual process that will be defined in Section \ref{combrw}.

\clearp

\section{Dual process and branching random walk}\label{dual}

\subsection{The sexual reproduction model}

\subsubsection{The graphical representation}\label{sexualdual}

We begin by constructing the process using a graphical representation that consists of a  collection of independent Poisson processes. Here, we give only a brief description of the construction.  More details can be found in Section 2a of \cite{DN94}. Define 
\beq\label{cstar}
c^*(\ep)= \sum_i \sup_{\xi^\ep\in  \{0,1\}^{\delta\ep\ZZ^d}} c^\ep_i(0, \xi^\ep)=(\beta+1)\ep^{-2},
\eeq
where $c^\ep_i(0, \xi^\ep)$ is the rate that the origin changes to state $i$  in the process $\xi^\ep_t$ when the configuration is $\xi^\ep$ .

\begin{itemize}
\item For every site $x\in \delta\ep\ZZ^d$ we have a Poisson process $\{T^{b,x}_n, n\geq 1\}$ with rate $c^*(\ep)$ and a sequence of i.i.d. random variables  $\{U^x_n,n\geq 1\}$ uniform on $(0,1)$. At time $T^{b,x}_n$ we use $U^x_n$ to determine the type of change that occurs:
\begin{enumerate}
\item If  $U^x_n\in(0,\frac{\beta}{1+\beta})$, $x$ gives birth to two particles on a randomly chosen pair from $x+{\cal N}_b$.
\begin{itemize}
\item If the state of $x$ is 1, then nothing occurs to the particle at $x$.
\item If the state of $x$ is 0, then $x$ flips to 1 if both of its children are 1's.
\end{itemize}
\item If $U^x_n\in (\frac{\beta}{1+\beta},1)$:
\begin{itemize}
\item If the state of $x$ is 0, then nothing occurs to the particle at $x$.
\item If the state of $x$ is 1, then $x$ flips to 0. 
\end{itemize}
\end{enumerate}

\item For every unordered pair $x,y \in \delta\ep\ZZ^d$ with $\|x-y\|_1=\delta\ep$ we assign a Poisson process $\{ T^{x,y}_n, n\geq 1\}$ with rate $(\delta\ep)^{-2}/2$. At an arrival of this Poisson process, the states of $x$ and $y$ are exchanged.
\end{itemize}

\mn
\subsubsection{The dual is almost a branching random walk.}\label{combrw}

For a particle at site $x\in \delta\ep\ZZ^d$ at time $t$, we denote by $\{\XX^\ep_s\}_{0\leq s\leq t}$ its dual process. The dual process is naturally defined only for $0\leq s\leq t$ but it is convenient to assume that the Poisson processes and uniform random variables in the construction are defined for negative times and hence define $\XX^\ep_s$ for all $s\geq 0$. We will focus on the case where $d=2$ in later discussion of the sexual reproduction model, but the comparison to a branching random walk in this section is general in all dimensions $d\geq 1$.

Let $R^\ep_0=0$ and let $R^\ep_m$ be the $m$-th time that a branching event occurs among the particles in $\XX^\ep_s$ and set $X^\ep_0(0)=x$ to represent the initial location of the first particle.

\begin{itemize}
\item
In between the branching time $\{R^\ep_m: m\geq 1\}$ the particles move by stirring. If there is a particle at $x$ or  $y$ at time $s$ and there is an arrival in $T^{x,y}_n$ at time $t-s$ then the particle at $x$ jumps to $y$ and a particle at $y$ jumps to $x$.

\item
At time $R^\ep_1$ if the branching occurs at $x_1$ we uniformly choose a pair of neighbors $x_1,x_2$ from 
$$x+\NN^\ep_b=x+\delta\ep \cdot \big\{\{e_1,e_2\}, \{-e_1, e_2\}, \{-e_1,-e_2\}, \{e_1,-e_2\} \big\},$$
add $x_1$ and $x_2$ to the dual, and number them as 1 and 2.

\item
At later branching times  $R^\ep_m$ if the branching occurs at $x_m$ then we add a randomly chosen pair from $x_m+{\cal N}^\ep_b$, and number the two new particles $2(m-1)+1$ and $2(m-1)+2$. A \textit{collision} is said to happen if a particle is born at the location already occupied by another particle. In this case the colliding particle is not added to $\XX^\ep$.  We also construct a (noncoalescing) branching random walk $\hat{\XX}^\ep$ in which two particles are always added, and if there is a collision an independent graphical representation is used to determine its movements. 
\end{itemize}

Notice that in the sexual reproduction model once a particle flips to state 0 its future is then independent of all its past. When constructing the dual process, once a particle flips to state 0 we don't necessarily need to probe into its past anymore. We can either treat this particle as ``dead" and do not let it branch again since we don't need information about its past, or we can let it branch despite its state so that the resulting dual process has the structure of a regular tree. Here we take the second treatment.

In order to have the probability of collisions in the dual process $\XX^\ep_t$ to be small we have chosen $\delta=\exp(-\ep^{-3})$, i.e., 
$\ep = (\log(1/\delta))^{-1/3}$,  so that $\delta \ll \ep$. Intuitively, if the stirring rate is large enough compared to the branching rate, then particles do not stay near each other for a sufficiently long time to have a birth that causes collisions. To simplify notation, we will write 
$$
\eta=\delta \ep
$$ 
from now on. In this notation, the dual process $\XX^\ep$ on $\eta\ZZ^d$ jumps at rate $2d\cdot \eta^{-2}/2$ to a randomly chosen neighbor.
   
\begin{lemma}\label{independence}
Let $T\in(0,\infty)$, $k\in\mathbb{N}$ and $x\in \RR^d$ be fixed. Let $\XX^\ep$ and $\hat{\XX^\ep}$ be defined as above and both start at $x$. There exists $\ep^*(k,T)>0$ so that for $\ep\in (0,\ep^*(k,T))$, 
$$P^\ep_x(\XX^\ep(t)=\hat{\XX^\ep}(t) \text{ for all } t\leq T)\geq 1-\ep^k.$$
\end{lemma}
\begin{proof}
This proof comes from Durrett and Neuhauser \cite{DN94}. To be self-contained we will present most of the details. We say a particle $X^\ep_k$ is \textit{crowded} at time $s$ if for some $j\neq k$, $\|X^\ep_k(s)-X^\ep_j(s)\|_1\leq \eta$.  To bound the number of collisions, we need to estimate the amount of time $X^\ep_k$ is crowded. Let $j\neq k$, $V^\ep_s=X^\ep_k(s)-X^\ep_j(s)$ and $W^\ep_s$ be a random walk that jumps to a randomly chosen neighbor at rate $2d\eta^{-2}$. Let $x,y\in \eta\{\pm e_1,\dots,\pm e_d\}$.  Then
\begin{center}
\begin{tabular}{ c c c } 

jumps from $x$ to & rate in $V$ & rate in $W$ \\ 
 $-x$ &  $\eta^{-2}/2$ & 0 \\ 
 0 & 0 & $ \eta^{-2}$ \\ 
 $x+y$& $\eta^{-2}$ &  $\eta^{-2}$\\
\end{tabular}
\end{center}
Since we are interested in $\|V^\ep_s\|_1$ we can ignore the first line in the above table, which does not change the norm. Then we can couple the jumps so that  $|\{s\leq t:\|V^\ep_s\|_1\leq \eta\}|$ is stochastically smaller than $w^{\eta}_t=|\{ s\leq t: \|W^\ep_s\|_1\leq \eta\}|$. Asymptotic results for random walks imply,  see (2.1) at page 301 of \cite{DN94}, that when $t\eta^{-2}\geq 2$,
$$
E w^{\eta}_t\leq
\begin{cases}
C\eta^2, &d\geq 3,\\
C\eta^2\log (\eta^{-2}),&d=2,\\
C\eta t^{1/2}, &d=1.
\end{cases}
$$
Let $\chi^k_\ep(t)$ be the amount of time $X^\ep_k$ is crowded in $[0,t]$ and $\KK_t$ be the total number of particles at time $t$. It is easy to see 
\begin{align*}
E(\chi^k_\ep(t)|\KK_t=K)&\leq K E w_t^{\eta},\\
E \KK_t&=\exp( \nu t), \text{ where }\nu=3c^*\ep^{-2},\\
E(\chi^k_\ep(t)) &\leq \exp(\nu t)E w^{\eta}_t.
\end{align*}
To see that with high probability no collisions occur, we note that the expected number of births from $X^\ep_k$ while there is some other $X^\ep_j$ in its neighborhood is (consider the worst case when $d=1$)
$$\leq E(\chi^k_\ep(t)) c^*\ep^{-2}\leq C_0 \eta t^{1/2}\ep^{-2} \exp(\nu t)$$
Take $K=\eta^{-0.2}$. Then $P(\KK_t>K)\leq K^{-1}\exp(\nu t)=\eta^{0.2}\exp(\nu t)$. When $\KK_t\leq K$, the expected number of collisions is smaller than 
$$KC_0 \eta t^{1/2}\ep^{-2} \exp(\nu t).$$
Combining the error probabilities we have the probability of a collision occurring before time $T$ is smaller than
\beq\label{probcollision}
\eta^{0.2}\exp(\nu T)+KC_0 \eta T^{1/2}\ep^{-2} \exp(\nu T)
\eeq
Since $\eta=\delta\ep=\ep\exp(-\ep^{-3})$ the above term vanishes as $\ep\to 0$ and decays faster than any polynomial of $\ep$. Then for any given $k\in\mathbb{N}$,  there exists some $\ep^*(k,T)$ so that when $\ep<\ep^*(k,T)$ the probability of collision \eqref{probcollision} is less than $\ep^k$. When there is no collision between $[0,T]$, $\XX^\ep(t)=\hat{\XX}^\ep(t)$ for all $0\leq t\leq T$.
\end{proof}

\subsubsection{Our random walks are close to Brownian motion}
Let $\hat{X}^\ep_t$ represent a single lineage in the comparison process $\hat{\XX}^\ep(t)$. At each branch point we will choose one lineage of the offspring particles to be $\hat{X}^\ep_t$ uniformly at random. We start by showing that the trajectory of a single lineage $\hat{X}^\ep_t$ of the dual process is close to a Brownian motion $W_t$ in $\RR^d$ when $\ep$ is small. Recall that $\hat{X}^\ep_t$ is a random walk that jumps at rate $d\eta^{-2}$ to a randomly chosen neighbor.

\begin{lemma}\label{err}
Let  $\hat{X}^\ep_t$ be a single lineage started at $x$ and let $k\in \mathbb{N}$. There exists some $\ep_0(k)$ and a coupling between the Brownian motion $W_t$ in $\RR^d$ and $\hat{X}^\ep_t$ so that for $\ep\in(0,\ep_0(k))$
$$
P( |W_{t}-\hat{X}^\ep_{t}|\geq \ep \text{ for some }t\leq k\ep^2|\log\ep|)\leq \ep^{2k}.
$$
\end{lemma}

\begin{proof}
Write $\hat{X}^\ep_{t}=( \hat{X}^{1,\ep}_t, \dots, \hat{X}^{d,\ep}_t)$ where for each $1\leq i\leq d$ $ \hat{X}^{i,\ep}_t$ is a random walk on $\eta \ZZ$ with rate $\eta^{-2}$. Let $\{N_i(t): 1\leq i\leq d\}$ be independent Poisson processes with rate $\eta^{-2}$ and let $Y^{(i)}_{1},Y^{(i)}_{2},\dots$ be i.i.d. random variables uniform on $\{-\eta,\eta\}$. For $1\leq i\leq d$, define discrete time random walks $S^{(i)}_n:=Y^{(i)}_{1}+Y^{(i)}_{2}+\cdots +Y^{(i)}_{n}$. Then we can observe that $S^{(i)}_{N_i(t)}$ has the same distribution as $\hat{X}^{i,\ep}_t$. Furthermore, if we let $N(t)=\sum_{i=1}^d N_i(t)$ then $S_{N(t)}:=(S^{(1)}_{N_1(t)},\dots, S^{(d)}_{N_d(t)})$ has then same distribution as $\hat{X}^{\ep}_t$.

From now on we consider the first coordinate $\hat{X}^{1,\ep}_t$ and $S^{(1)}_{N_1(t)}$ of the two random walks. Write $x=(x_1,\dots,x_d)$. By Skorohod's embedding theorem (see \cite{OK97} Theorem 12.1), there is a Brownian motion $B_t$ in $\RR$ started at $x_1$ and a sequence of stopping times $0=\tau_0\leq \tau_1\leq \dots$ such that $B(\tau_i)=S^{(1)}_i$. Moreover, the differences $\tau_i-\tau_{i-1}$ are i.i.d. with 
$E(\tau_i-\tau_{i-1})=E|Y^{(1)}_1|^2=\eta^2$ and $E(\tau_i-\tau_{i-1})^2\leq 4E|Y^{(1)}_1|^4=4\eta^4$.

Note that 
$$
\tau_{N_1(t)}- t= (\tau_{N_1(t)}-N_1(t)E\tau_1)+(N_1(t)E\tau_1-t)
$$
is a martingale, so $L_2$-maximal inequality implies 
$$
E\left(\max_{0\leq s\leq t} |\tau_{N_1(s)}-s|^2\right)
\leq 4E(\tau_{N_1(t)}-t)^2\leq 4E[N_1(t)]\cdot Var(\tau_1) \leq 16t\eta^2. 
$$
By Chebyshev's inequality, 
\beq\label{tau}
P\left(\max_{0\leq s\leq t}|\tau_{N_1(s)}-s|\geq \eta^{1/2}\right)
\leq \eta^{-1} E\left(\max_{0\leq s\leq t} |\tau_{N_1(s)}-\sigma^2 s|^2\right)\leq 16t\eta. 
\eeq

Write $W_t=(W^{(1)}_t,\dots,W^{(d)}_t)$. Since $W^{(1)}_t$ is itself a one dimensional Brownian motion, without loss of generality we can let $W^{(1)}_t=B_t$. Notice that $|\hat{X}^{1,\ep}_t-W^{(1)}_t|=|S^{(1)}_{N_1(t)}-B(t)|=|B(\tau_{N_1(t)})-B(t)|$. Then applying \eqref{tau}
\begin{align*}
&P(|W^{(1)}_{t}-\hat{X}^{1,\ep}_t|\geq \ep/\sqrt{d} \text{ for some }t\leq k\ep^2|\log \ep|)\\
=& P(\max_{0\leq t\leq k\ep^2|\log \ep|}|B_{\tau_{N_1(t)}}-B_{t}|\geq \ep/\sqrt{d})\\
 \leq& P(\max_{0\leq t\leq k\ep^2|\log \ep|}|\tau_{N_1(t)}-t|\geq \eta^{1/2})\\
 &+P(\max_{0\leq t\leq k\ep^2|\log \ep|}|\tau_{N_1(t)}-t|< \eta^{1/2}, \max_{0\leq t\leq k\ep^2|\log \ep|}|B_{\tau_{N_1(t)}}-B_{t}|\geq \ep/\sqrt{d})\\
 \leq&16 k\ep^2|\log \ep|\eta+P(N_1(k\ep^2|\log \ep|)>\eta^{-2.5}k\ep^2|\log \ep|)\\
 & +\eta^{-2.5}k\ep^2|\log \ep| \cdot P( \sup_{s\in [-\eta^{1/2},\eta^{1/2}]} |B(s)-B(0)|\geq \eta^{1/6}).
\end{align*}
The second term is upper bounded by $\eta^{1/2}$ due to Markov inequality. To estimate the third term, let $Z$ be a standard Gaussian variable. By reflection principle,
\begin{align*}
P(\sup_{s\in [-\eta^{1/2},\eta^{1/2}]} |B(s)-B(0)|\geq \eta^{1/6})&\leq  2P(\sup_{s\in [0,\eta^{1/2}]} |B(s)-B(0)|\geq \eta^{1/6}) \\
&\leq 4P(|B(\eta^{1/2})-B(0)|\geq \eta^{1/6})=4P( \eta^{1/4}Z\geq \eta^{1/6})\\
&\leq 4e^{-\eta^{-1/6}/4}.
\end{align*}
Thus 
\begin{align}\label{onedimdiff}
\nonumber&P(|W^{(1)}_{t}-\hat{X}^{1,\ep}_t|\geq \ep/\sqrt{d} \text{ for some }t\leq k\ep^2|\log \ep|) \\
\leq &16k\ep^2|\log \ep|\eta+\eta^{1/2}+4\eta^{-2.5}k\ep^2|\log \ep|e^{-\eta^{-1/6}/4}\leq C\eta^{1/2}
\end{align}
for some $C>0$. Finally, it follows from \eqref{onedimdiff} that
\begin{align*}
P(|W_{t}-\hat{X}^{\ep}_t|\geq \ep \text{ for some }t\leq k\ep^2|\log \ep|)&\leq d\cdot P(|W^{(1)}_{t}-\hat{X}^{1,\ep}_t|\geq \ep/\sqrt{d} \text{ for some }t\leq k\ep^2|\log \ep|)\\
&\leq dC\eta^{1/2}\leq \ep^{2k}.
\end{align*}
\end{proof}


\subsection{Voter model perturbations}\label{vmp}

\subsubsection{The dual is close to a branching random walk}\label{voterdual}
The dual process $\XX^\ep$ is a coalescing branching random walk. As a result of the coalescence, the dual process does not have the tree structure that leads to independence among subtrees. The situation is not too bad once we realize coalescence mostly happens between particles with the same parent in a short amount of time after their births. Hence we will construct a comparison process $\hat{\XX^\ep}$ that has the desired tree structure.

Recall that the voting and branching neighborhoods are
$$
{\cal N}^\ep_v = \{-\eta, \eta\}^d \quad \text{ and }\quad \NN^\ep_b=[-\eta L,\eta L]^d \cap \eta \ZZ^d
$$
for a fixed $L$.
Let $J(t)$ denote the set of particles in $\XX^\ep$ at time $t$. If two particles $i$ and $j$ coalesce at time $s$, then $i\vee j$ is removed from $J(s-)$ to form $J(s)$. Set $R_0=0$ and let $R_m$ be the $m$-th branching time in $\XX^\ep$. Similarly, define $\hat{J}(t)$ and $\hat{R}_m$ for the process $\hat{\XX}^\ep$. 

The comparison process $\hat{\XX}^\ep$ is constructed as follows:
\begin{itemize}
\item At time $\hat{R}_m$, the parent particle at $x$ gives birth to $N_0=4$ particles at sites $(Y_1,\dots,Y_4)$ chosen uniformly without replacement from $x+\NN^\ep_b$. 
\item During $[\hat{R}_m, \hat{R}_m+\eta^{1/2})$ we do not allow birth events. The particles move as coalescing random walks in $\eta\ZZ^d$ and we allow the particles within the new family (i.e. the parent particle and its $N_0$ children) to coalesce with each other. 
\item During $[\hat{R}_m+\eta^{1/2},\hat{R}_{m+1})$ the particles move as random walks without coalescing and give births at rate $\ep^{-2}$.
\end{itemize}
If we view the interval $[\hat{R}_m, \hat{R}_m+\eta^{1/2})$ as one single point in time then the process $\hat{\XX}^\ep$ would have the desired tree structure where each vertex has a random number of offspring depending on the coalescence. Note that 
$$\hat{R}_{m+1}-\hat{R}_m \overset{d}{=} \sqrt{\eta}+\text{Exponential}(\ep^{-2} \hat{J}(\hat{R}_m+\sqrt{\eta})).$$

The graphical representations of $\XX^\ep$ and $\hat{\XX}^\ep$ can be coupled until there is a coalescence in $\XX^\ep$ that is not in $\hat{\XX}^\ep$. Whenever this happens we use an independent graphical representation to determine the movement of the non-coalesced particle in $\hat{\XX}^\ep$. We hope to couple $\XX^\ep$ and $\hat{\XX}^\ep$ in a way that the former is dominated by the latter. The obstacles in doing so are (i) $\XX^\ep$ can have births during intervals $\{[\hat{R}_m, \hat{R}_m+\eta^{1/2}): m\geq 1\}$ (ii) if the scenario in (i) does not happen, since $\hat{\XX}^\ep$ has more particles ever since the first coalescence in $\XX^\ep$ that is not in $\hat{\XX}^\ep$, the branching times $\hat{R}_m$ could arrive faster than $R_m$. As we will prove soon, both (i) and (ii) will not be the case with high probability. Our goal is to establish the following coupling between $\XX^\ep$ and $\hat{\XX}^\ep$ in such a way that the former is dominated by the latter.

\begin{lemma}\label{voterinde}
Let $T\in(0,\infty)$, $k\in\mathbb{N}$ and $x\in \RR^d$ be fixed. Let $\XX^\ep$ and $\hat{\XX}^\ep$ be defined as above and both start at $x$. There exists $\ep^*(k,T)>0$ so that for $\ep\in (0,\ep^*(k,T))$, 
$$P^\ep_x(\XX^\ep(t)=\hat{\XX}^\ep(t) \text{ for all } t\leq T)\geq 1-\ep^k.$$
\end{lemma}

\begin{proof}
Let $N_{T}=\min\{m: R_m>{T}\}$ and  define the good events
\begin{align*}
&G_1=\{ R_m-R_{m-1}>\sqrt{\eta} \hh\text{ for all }1\leq m\leq  N_{T}\}\\
&G_2=\{ R_m=\hat{R}_m \hh\text{ for all }1\leq m\leq  N_{T}\}\\
&G_3=\{ J(s)=J(R_{m-1}+\sqrt{\eta})  \text{ for all }s\in [R_{m-1}+\sqrt{\eta},R_m) \text{ and all }1\leq m\leq  N_{T}\}.
\end{align*}
Observe that on $G\equiv G_1\cap G_2\cap G_3$ we can couple $\XX^\ep$ and $\hat{\XX}^\ep$ exactly. Hence it suffices to upper bound $P(G^c)$. The estimates have already been done in detail in \cite{CDP13} so we will cite the relevant results instead of repeating the arguments. 

Lemma 2.4 in \cite{CDP13} gives 
$$P(G_1^c)=P(\min_{1\leq m\leq N_{T}} R_m-R_{m-1}\leq \sqrt{\eta})\leq \eta^{1/8}.$$

Let $X^{\ep,j}_s$ denote the location of particle $j$ in $\XX^\ep$ at time $s$. Define 
$$\tau_m=\inf \{ s\geq R_{m-1}+\sqrt{\eta}: \inf_{i\neq j\in J(s)}: |X^{\ep,i}_s-X^{\ep,j}_s|\leq \eta^{7/8}\}$$
 Lemma 2.7 in \cite{CDP13} gives 
$$P(G_3^c)=P(\tau_m<R_m \text{ for some }1\leq m\leq N_{T})\leq \eta^{1/16}.$$
The memoryless property of exponential random variables implies that  
$$(R_{m+1}-R_m | G_1\cap G_3) \overset{d}{=} \sqrt{\eta}+\text{Exponential}( J(R_m+\sqrt{\eta})).$$
We will argue by induction that $G_1\cap G_3\subseteq G_2$. First note $R_0=\hat{R}_0=0$. Suppose $R_m=
\hat{R}_m$ holds up to $m=k$ on $G_1\cap G_3$. Then we should have $J(R_k+\sqrt{\eta})=\hat{J}(\hat{R}_k+\sqrt{\eta})$ on $G_1\cap G_3$. This means
\begin{align*}
(R_{k+1}-R_k | G_1\cap G_3)& \overset{d}{=} \sqrt{\eta}+\text{Exponential}( \hat{J}(\hat{R}_k+\sqrt{\eta}))\\
& \overset{d}{=} \hat{R}_{k+1}-\hat{R}_k
\end{align*}
Therefore $R_{k+1}=\hat{R}_{k+1}$ on $G_1\cap G_3$ and this concludes the proof of $G_1\cap G_3\subseteq G_2$.
Finally,
$$P(G^c)\leq P(G_1^c)+P(G_3^c)\leq \eta^{1/8}+\eta^{1/16}\leq \ep^k$$
for any $k\in\mathbb{N}$ when $\ep$ is sufficiently small.
\end{proof}

\subsubsection{Our random walks are almost Brownian motions}

We will show the trajectory of a single lineage $\hat{X}^\ep_t$ of the dual process is close to a Brownian motion $W_t$ in $\RR^d$. Note that $\hat{X}^\ep_t$ is a random walk in $\eta \ZZ^d$ that jumps at rate $\eta^{-2}$ to a site chosen uniformly random from its neighborhood of the form $\NN^\ep_v=\{-\eta,\eta\}^d$. The following lemma is essentially the same as Lemma \ref{err}. Note that the random walk in Lemma \ref{err} has jump rate $2d\cdot \eta^{-2}/2$ while here the random walk has jump rate $\eta^{-2}$, implying that $\hat{X}^\ep_t$ would converge to a time-changed Brownian motion. The proof is essentially the same as that of Lemma \ref{err} and hence is omitted.

\begin{lemma}\label{errv}
Let  $\hat{X}^\ep_t$ be a single lineage started at $x$ and let $k\in \mathbb{N}$. There exists $\ep_0(k)$ and a coupling of the Brownian motion $W_t$ in $\RR^d$ and $\hat{X}^\ep_t$ so that for $\ep\in(0,\ep_0(k))$
$$P( |W_{\sigma^2 t}-\hat{X}^\ep_{t}|\geq \ep \text{ for some }t\leq k\ep^2|\log\ep|)\leq \ep^{2k}$$
for $\sigma=1/\sqrt{d}$.
\end{lemma}

\subsection{Computing the state of $x$ at time $t$}\label{compp}

To do this, we use the comparison process $\hat{\XX}^\ep$ constructed in Section \ref{sexualdual} and Section \ref{voterdual} and work backwards in time. $\hat{\XX}^\ep$ has a tree structure so we can follow \cite{Etheridge} to define a \textit{time-labelled tree} $\TT(\hat{\XX}^\ep(t))$ for $\hat{\XX}^\ep$. Since $\hat{\XX}^\ep$ and $\hat{\XX}$ has the same tree structure except for the rescaling, to simplify notation we consider $\TT(\hat{\XX}(t))$ from now on.

Each branch point in $\{\hat{\XX}\}_{0\leq s\leq t}$ is a vertex in the tree $\TT(\hat{\XX}(t))$ and is assigned a time label $t_v$ that is the corresponding branching time in $\hat{\XX}$. For the sexual reproduction model, at each branch point the parent gives birth to $N_0=2$ children, so the tree $\TT(\hat{\XX}(t))$ branches into 3 lineages. For the voter model perturbations, at each branch point the parent gives birth to $N_0=4$ children some of whom will coalesce into one. There are two ways to look at $\TT(\hat{\XX}(t))$: we can either see it as a Galton-Watson tree where the offspring distribution is determined by the coalescence, or we can still see it as a regular tree where each vertex has $N_0+1$ children and deal with the influence of coalescence in a computing process that will be introduced later as an \textit{algorithm}. Here we take the second approach.

Now we will describe an \textit{algorithm} that computes the state of $x$ at time $t$ given the graphical representation and the initial states of the particles in $\hat{\XX}(t)$.  Since we are considering the dual process without rescaling, with a little abuse of notation let $p: \ZZ^d\to [0,1]$ be the initial condition.
 
\mn
\textbf{Algorithm for sexual reproduction model with fast stirring:} 
\begin{enumerate}
\item Each particle $i$ in $\TT(\hat{\XX}(t))$ is independently assigned state 1 with probability $p(\hat{X}^i_t)$ and  state 0 with probability $1-p(\hat{X}^i_t)$.
\item At each branch point $v$ in $\TT(\hat{\XX}(t))$, we have an independent random variable $U_v$ uniform on $(0,1)$ that determines the state of the parent particle according to rules specified in Section \ref{sexualdual}.
\end{enumerate}

 \mn
\textbf{Algorithm for nonlinear voter models:}
Let $\{\pi_v\}$ be a collection of i.i.d. random partition of the set $\{0, 1,\dots, N_0\}$, where $v$ represents a vertex in the $N_0+1$ regular time-labelled tree $\TT(\hat{\XX}(t))$. The law of $\pi_v$ is given by the coalescence of particles within the same family within time $\sqrt{\eta}$ after birth.
 
\begin{enumerate}
\item Each particle $i$ in $\TT(\hat{\XX}(t))$ is independently assigned state 1 with probability $p(\hat{X}^i_t)$ and  state 0 with probability $1-p(\hat{X}^i_t)$.
\item At each branch point $v$ in $\TT(\hat{\XX}(t))$, we first sample a random partition $\pi_v$. For vertices in the same cell of $\pi_v$, we uniformly choose one of them and let its state be the state of every vertex in that cell. Let $i_1$ denote the total number of $1$'s among these $N_0+1$ particles. Then an independent random variable $U_v$ uniform on $(0,1)$ is sampled. If $U_v<a_{i_1}$ then set the output to be 1, otherwise set the output to be 0.
\end{enumerate}

For Lotka-Volterra systems, since the effective perturbations only occur when there is no coalescence among the three children, see \eqref{gLV}, we can consider only such branch points and effectively reduce the branching rate to $\theta p_3\ep^{-2}$. At each branch point, the state of the chosen lineage only flips when it is opposite to both of the other lineages. This is essentially performing a majority vote, which is why \eqref{gLV} is the same as \eqref{g0}. Hence the proof for Lotka-Volterra systems is the same as that in \cite{Etheridge}. 
 
\mn
\textbf{Algorithm for Lotka-Volterra systems:}  
\begin{enumerate}
\item Each particle $i$ in $\TT(\hat{\XX}(t))$ is independently assigned state 1 with probability $p(\hat{X}^i_t)$ and  state 0 with probability $1-p(\hat{X}^i_t)$.
\item Let the branching event occur at rate $\theta p_3\ep^{-2}$. At each branch point $v$ in $\TT(\hat{\XX}(t))$, we perform a majority vote.
\end{enumerate}

Starting from states of the leaves of $\TT(\hat{\XX}(t))$, the above algorithms compute the state of the root at $x$. From now on we use use $\VV_p(\hat{\XX}(t))$ to denote the output, i.e., the state of the root of $\TT(\hat{\XX}(t))$. Note that for a branching Brownian motion $\WW_t$ in $\RR^d$ we can define $\VV_p(\WW_t)$ in the same way except that the initial condition $p$ will be defined on $\RR^d$ instead of $\ZZ^d$.

\clearp

\section{Convergence to motion by mean curvature}\label{mbymc}

Here we will prove the result assuming the $g$ function has properties (G0)-(G5).
In the next section we will check those conditions in our examples. 
A second consequence of concavity for $p\in(u_0, u_+)$ is that if $p\in[u_0+\eta,u_+ - \eta]$
\beq
g(p+\eta)-2g(p)+g(p-\eta)\leq 0.
\label{G5}
\eeq
To prove \eqref{G5}, we note that
$$
\int_{p-\eta}^p \int_x^{x+\eta} g''(y) \, dy \,dx = g(p+\eta)-2g(p)+g(p-\eta).
$$

\subsection{Branching Brownian motion in one dimension}\label{bbmd1}

Define the initial condition $p_0: \RR \to [0,1]$ to be $p_0(x)=u_+\cdot1_{\{x\geq 0\}}+u_-\cdot1_{\{x< 0\}}$ and write $\VV:=\VV_{p_0}$.
In this section we will consider one dimensional branching Brownian motion $\BB_t$, beginning by listing the useful properties of $\VV(\BB(t))$.

\mn
\textbf{Monotonicity.} When the interaction rule is attractive and the initial condition $p_0$ is nondecreasing in $x$ so for any $x_1\leq x_2\in \RR$,
$$
P^\ep_{x_1}[\VV(\BB(t))=1]\leq P^\ep_{x_2}[\VV(\BB(t))=1].
$$

\mn
\textbf{Antisymmetry.}
We use $\TT(\BB(t))$ to denote the time-labelled tree for $\BB_t$ and write 
$$
P^t_x(\TT)=P^\ep_x(\VV(\BB(t))=1| \TT(\BB(t))=\TT).
$$
Applying the reflection from $z$ to $-z$, and using the symmetry of the Brownian motion conditioned on $\{\TT(\BB(t))=\TT\}$, we see that for any time-labelled tree $\TT$ 
$$
P^t_z(\TT)=2u_0-P^t_{-z}(\TT).
$$
The last property implies $P^t_0(\TT)=u_0$. Using monotonicity we have
$$
P^t_z(\TT)\geq u_0 \quad \text{ for $z\geq 0$}, \quad P^t_z(\TT)\leq u_0 \quad \text{ for $z\leq 0$}.$$

\subsubsection{Useful inequalities}

So far the function $g: [0,1]\to [0,1]$ has a single variable. It is natural to extend $g$ to be a function on $[0,1]^{N_0+1}$. Let $(p_1,\dots, p_{N_0+1})\in [0,1]^{N_0+1}$ and $g(p_1,\dots,p_{N_0+1})$ is the probability that the output at the branch point is 1 when the inputs are independent Bernoulli random variables with rate $p_1,\dots, p_{N_0+1}$ respectively. With a slight abuse of notation, we will use $g(\cdot)$ to stand for both.

\begin{lemma}\label{1dbm}
For any time-labelled tree $\TT$, and time $t>0$ and any $z\geq 0$,
$$
P^t_z(\TT)\geq u_+ P_z(B_t\geq 0)+u_-P_z(B_t<0).
$$
\end{lemma}

\begin{proof}
The proof is by induction on the number of branching events in the tree $\TT$.
Suppose time $\tau$ is the first branching event in $\TT$ and that the subtrees corresponding to the $N_0+1$ offspring are $\TT_1,\dots,\TT_{N_0+1}$. Letting 
\begin{align*}
P^t_z(\TT*)&=(P^t_z(\TT_1),\dots,P^t_z(\TT_{N_0+1})).\\
h(p_1,\dots ,p_{N_0+1})&=g(p_1,\dots ,p_{N_0+1})-\frac{1}{N_0+1}(p_1+\dots+p_{N_0+1}).
\end{align*}
we can write
\begin{align*}
P^t_z(\TT)
&=E_z(g(P^{t-\tau}_{B_\tau}(\TT*))=E_z(g(P^{t-\tau}_{B_\tau}(\TT_1),\dots,P^{t-\tau}_{B_\tau}(\TT_{N_0+1})))\\
&=E_z(h(P^{t-\tau}_{B_\tau}(\TT_1),\dots,P^{t-\tau}_{B_\tau}(\TT_{N_0+1})))+\frac{1}{N_0+1}\sum_{i=1}^{N_0+1} E_z(P^{t-\tau}_{B_\tau}(\TT_i))
\end{align*}
Write $h(p)=h(p,\dots,p)$. Observe that $h(u_+-p)=-h(u_-+p)$ due to (G1), which implies
\beq\label{heq}
h(P^{t}_{-z}(\TT*))=h(2u_0-P^t_z(\TT*))=h(u_+-(-u_-+P^t_z(\TT*)))=-h(P^t_z(\TT*)).
\eeq
It follows that 
\begin{align*}
E_z(h(P^{t-\tau}_{B_\tau}(\TT*))
&=E_z(h(P^{t-\tau}_{B_\tau}(\TT*))(1_{\{B_\tau\geq 0\}}+1_{\{B_\tau< 0\}})\\
&=E_z(h(P^{t-\tau}_{B_\tau}(\TT*))1_{\{B_\tau\geq 0\}})-E_z(h(P^{t-\tau}_{-B_\tau}(\TT*))1_{\{B_\tau< 0\}})\quad \text{ (by (\ref{heq}))}\\
&=\int_0^\infty h(P^{t-\tau}_x(\TT*))( \phi_{z,\tau}(x)-\phi_{z,\tau}(-x))\;dx
\end{align*}
where $\phi_{z,t}(x)$ is the probability density function of a Brownian motion starting at site $z$ at time t. Since $P^{t-\tau}_x(\TT_i)\geq u_0$ for $x\geq 0$ we have $h(P^{t-\tau}_x(\TT*))\geq 0$. Spatial symmetry of Brownian motion and the fact that $\phi_{z,t}(x)$ is decreasing on $x\ge z$ implies $\phi_{z,\tau}(x)-\phi_{z,\tau}(-x)\geq 0$ for all $x\geq 0$. That is, $E_z(h(P^{t-\tau}_{B_\tau}(\TT*))\geq 0$.

For $i=1,\dots, N_0+1$, by the induction hypothesis
\begin{align*}
E_z(P^{t-\tau}_{B_\tau}(\TT_i))
&\geq u_+E_z( P_{B_\tau}(B_{t-\tau}\geq 0))+u_-E_z( P_{B_\tau}(B_{t-\tau}< 0))\\
&=u_+P_z(B_t\geq 0)+u_-P_z(B_t< 0).
\end{align*}
If follows that 
$$E_z(g(P^{t-\tau}_{B_\tau}(\TT*))\geq \frac{1}{N_0+1}\sum_{i=1}^{N_0+1} E_z( P^{t-\tau}_{B_\tau}(\TT_i))\geq u_+P_z(B_t\geq 0)+u_-P_z(B_t< 0).$$
\end{proof}

We define the iterates of $g$, $g^{(n)}(p)$, by
$$
g^{(n)}(p)=g( g^{(n-1)}(p)), \quad g^{(1)}(p)=g(p).
$$
The fixed points at $u_-$ and $u_+$ of $g$ are attracting and $u_0$ is unstable. That is, if we start from $u_0+\ep$, then iterating $g$ will lead to $u_+$ while if we start at $u_0-\ep$, iterating $g$ will take us down to $u_-$. Lemma \ref{lemma2.8} quantifies the rate of convergence.

\begin{lemma}\label{lemma2.8}
For all $k\in \mathbb{N}$ there exists $A(k)<\infty$ such that, for all $\ep\in(0,u_+-u_0-\delta_*)$ where $\delta_*$ is defined in (G4) and $n\geq A(k)|\log \ep|$ we have 
$$g^{(n)}(u_0+\ep)\geq u_+-\ep^k \quad \text{ and }\quad g^{(n)}(u_0-\ep)\leq u_--\ep^k.$$
\end{lemma}

\begin{proof}
(G4) (i.e., \eqref{G4}) implies that if $\delta<\delta_*$ then $u_+-g(u_+-\delta) \le (1-c_0)\delta$. Iterating gives
$$
u_+-g^{(n)}(u_+-\delta)\leq (1-c_0)^n(u_+-\delta).
$$
That is, there is some constant $C_k$ such that if $\delta<\delta_*$ then for  $n\geq C_k|\log \ep|$ we have 
$$
g^{(n)}(u_+-\delta)\geq u_+-\ep^k.
$$
It remains to find an $M_\ep$, which will depend on $\ep$, so that $g^{(M_\ep)}(u_0+\ep)\geq u_+-\delta_*$. 

By \eqref{G2} we know $g'(u_0)>1$. Since $u_0$ and $u_+$ are two fixed points of $g$ and $g$ is strictly increasing, we have $g(p)>p$ for $p\in (u_0,u_+-\delta_*]$. It follows that 
$$
k_1\equiv \inf_{x\in (0,u_+-u_0-\delta_*]} \frac{g(u_0+x) - (u_0+x)}{x} > 0
$$
so for $x\in 
[u_0+\ep,u_+-\delta_*]$ we have $g(u_0+x) - u_0 \ge (1+k_1)x$. Hence for $m\in \mathbb{N}$ such that  $g^{(m)}(u_0+\ep) < u_+-\delta_*$ we have $g^{(m)}(u_0+\ep) \ge u_0+(1+k_1)^m\ep$. This implies we can take $M_\ep=B|\log\ep|$ where $B=1/\log(1+k_1)$. Taking $A(k) = B+C_k$ completes the proof.
\end{proof}

Since the branching rate $c^*\ep^{-2}$ is large when $\ep$ is small, then even for  a small $t$ the tree $\TT(\BB(t))$ should be have a lot of vertices. For $l\in\RR$, let $\TT^{reg}_l$ denote a ternary tree with depth $\lceil l \rceil$. For a time-labelled ternary tree $\TT$, we write $\TT\supseteq \TT^{reg}_l$ if $\TT^{reg}_l$ can be embedded in $\TT$ as a subtree. The next two results are Lemma 2.10 and 2.11 in \cite{Etheridge}. The proofs are exactly  the same so they are omitted. 

\begin{lemma}\label{lem2.9}
Let $k\in\mathbb{N}$ and let $A=A(k)$ be as in Lemma \ref{lemma2.8}. Then there exists $a_1=a_1(k)$ and $\ep_1=\ep_1(k)$ such that, for all $\ep\in(0,\ep_1)$ and $t\geq a_1\ep^2|\log \ep|$,
$$P^\ep[ \TT(\BB(t))\supseteq \TT^{reg}_{A(k)|\log \ep|}]\geq 1-\ep^k.$$
\end{lemma}

\begin{lemma}\label{lem2.10}
Let $k\in\mathbb{N}$, and let $a_1(k)$ as in Lemma \ref{lem2.9}. Then there exists $d_1(k), \ep_1(k)$ such that for all $\ep\in(0,\ep_1(k))$ and all $s\leq a_1\ep^2|\log \ep|$,
$$P^\ep_x[\exists i\in N(s): |B_i(s)-x|\geq d_1(k)\ep|\log \ep|] \leq \ep^k,$$
where $N(s)$ is the set of indices of particles in $\BB$ up to time $s$.
\end{lemma}

\mn
While the proof of Lemma \ref{lemma2.8} is fresh on the reader's mind we will prove

\begin{lemma}\label{upden}
For a fixed $k\in\mathbb{N}$, there exists $\sigma_1(k)>0$ such that for $t\geq \sigma_1(k)\ep^2|\log \ep|$ and $x\in \RR$
$$
P^\ep_x[\VV_p(\WW(t))=1]\leq u_++\ep^k
$$
where $p: \RR^d \to [0,1]$ is the initial condition satisfying (G0).

\end{lemma}
\mn
\textbf{Remark.} \textit{The same conclusion also holds for $P^\ep_x[\VV_p(\hat{\XX^\ep}(t))=1]$} following the same proof.

\begin{proof}
First we consider the case where 1 is not a fixed point of $g$. Since $u_+$ is a fixed point of $g$ and $g'(u_+)<1$ by (G2), it is easy to see $g(p)<p$ on $(u_+,1]$. It follows that 
$$k_2\equiv \inf_{x\in (0,1-u_+]} \frac{ (u_++x)-g(u_++x)}{x}\in (0,1),$$
which implies that if $\delta\in [0,1-u_+]$
$$
g(u_++\delta) - u_+  \le (1-k_2)\delta.
$$
Iterating as in the proof of Lemma \ref{lemma2.8} 
$$
g^{(n)}(u_++\delta)-u_+\leq (1-k_2)\left(g^{(n-1)}(u_++\delta)-u_+\right)\leq (1-k_2)^n \delta.
$$
By assumption (G0), since 1 is not a stable fixed point of $g$ we get the largest value by setting $p\equiv1$. In order to have $g^{(n)}(1)\leq u_++\ep^k$ we need 
$$
g^{(n)}(u_++(1-u_+))-u_+\leq (1-k_2)^n(1-u_+)\leq \ep^k.
$$
It is easy to see that there exists $C(k)>0$ such that the above inequality holds for $n\geq C(k)|\log \ep|$.
It follows from Lemma \ref{lem2.9} that there exists $\sigma_1(k)>0$ such that for $t\geq \sigma_1(k)\ep^2|\log \ep|$
$$P^\ep[ \TT(\WW(t))\supseteq \TT^{reg}_{C(k)|\log \ep|}]\geq 1-\ep^k.$$
Therefore, when $t\geq  \sigma_1\ep^2|\log \ep|$
$P^\ep_x[ \VV_p(\WW(t))=1]\leq u_++\ep^k+\ep^k=u_++2\ep^k.$

The second case where 1 is a fixed point of $g$ follows similarly. By assumption (G0) we can set $p\equiv 1-\varep$ for some arbitrarily small $\varep>0$. Modify the definition of $k_2$ to be 
$$k_2(\varep)\equiv \inf_{x\in (0,1-\varep-u_+]} \frac{ (u_++x)-g(u_++x)}{x}\in (0,1).$$
In order to have $g^{(n)}(1-\varep)\leq u_++\ep^k$ we need 
$$
g^{(n)}(u_++(1-\varep-u_+))-u_+\leq (1-k_2)^n(1-\varep-u_+)\leq \ep^k.
$$
The rest of the argument for the second case is the same.
\end{proof}

\subsubsection{The main result in one dimension}

We are now ready to prove  

\begin{theorem}\label{thm2.5}
Fix any $T^*\in(0,\infty)$. For all $k\in \mathbb{N}$ there exist $c_1(k)$ and $\ep_1(k)>0$ such that, for all $t\in [0,T^*]$ and all $\ep\in (0,\ep_1)$,
\begin{enumerate}
\item for $x\geq c_1(k)\ep|\log \ep|$, we have $P^\ep_x[\VV(\BB(t))=1]\geq u_+-\ep^k$,
\item for $x\leq -c_1(k)\ep|\log \ep|$, we have $P^\ep_x[\VV(\BB(t))=1]\leq u_-+\ep^k$,
\end{enumerate}
\end{theorem}

\mn\textit{Proof of Theorem \ref{thm2.5}}.
For all $\ep<1/2$, define $z_\ep$ implicitly by the relation 
\beq\label{zep}
P_0(B_{T^*}\geq -z_\ep)=1/2+(u_+-u_-)^{-1}\ep
\eeq
and note that $z_\ep\sim (u_+-u_-)^{-1}\ep\sqrt{2\pi T^*}$ as $\ep \to 0$. Let $\ep_1(k)<1/2$ be sufficiently small so that Lemma \ref{lem2.9} and \ref{lem2.10} hold for $\ep\in (0,\ep_1)$. Let $d_1(k)$ be given by Lemma \ref{lem2.10} and let $c_1(k)=2d_1(k)$ so that, for $\ep\in(0,\ep_1)$,
$$d_1(k)\ep|\log\ep|+z_\ep \leq c_1(k)\ep |\log \ep|.$$
Let $a_1(k)$ be given by Lemma \ref{lem2.9} and let $\delta_1=\delta_1(k,\ep)=a_1(k)\ep^2|\log \ep|$. 

Note that $g(u_+)=u_+$, which means if we start with initial condition $p(x)\equiv u_+$ then
\beq\label{Vphi}
P^\ep_z(\VV_p(\BB(t))=1)=u_+ \quad \text{ for all $t>0, z\in \RR$}.
\eeq
If $t\in (0,\delta_1)$ and $z\geq c_1\ep|\log \ep|$, then Lemma \ref{lem2.10} and \eqref{Vphi} gives
\begin{align*}
P^\ep_z(\VV(\BB(t))=0)&\leq P^\ep_z( \exists i \in N(t) \text{ such that }|B_i(t)-z|\geq d_1\ep|\log \ep|)+P^\ep_z(\VV_p(\BB(t))=0)\\
&\leq 1-u_++\ep^k.
\end{align*}
We now suppose that $t\in [\delta_1,T^*]$ and $z\geq c_1\ep|\log \ep|$, and define
$$p_{t-\delta_1}(z)=P^\ep_z(\VV(\BB(t-\delta_1))=1),$$
and let  $\psi^\ep\equiv p_{t-\delta_1}(z_\ep).$
Write $\{\BB(\delta_1)>z_\ep\}$ for the event $B_i(\delta_1)>z_\ep$ for all $i\in N(\delta_1)$. Then 
\begin{align*}
P^\ep_z(\VV(\BB(t))=1)&=P^\ep_z(\VV_{p_{t-\delta_1}}(\BB(\delta_1))=1)\\
&\geq P^\ep_z\left(\{\VV_{\psi^\ep}(\BB(\delta_1))=1\} \cap \{\BB(\delta_1)>z_\ep\}\right)\\
&\geq  P^\ep_z(\VV_{\psi^\ep}(\BB(\delta_1))=1)-\ep^k
\end{align*}
By definition of $z_\ep$ in \eqref{zep} and $t-\delta_1<T^*$,
\begin{align*}
\psi^\ep=P^\ep_{z_\ep}(\VV(\BB(t-\delta_1))=1)&\geq u_+P_{z_\ep}(B_{t-\delta_1}\geq 0)+u_-P_{z_\ep}(B_{t-\delta_1}< 0)\\
&=u_+(1/2+(u_+-u_-)^{-1}\ep)+u_-(1/2-(u_+-u_-)^{-1}\ep)=u_0+\ep.
\end{align*}
It follows from Lemma \ref{lemma2.8} and \ref{lem2.9} that
\begin{align*}
P^\ep_z(\VV_{\psi^\ep}(\BB(\delta_1))=1)&\geq g^{(A(k)|\log \ep|)}(u_0+\ep)P^\ep\left( \TT(\BB(t))\supseteq \TT^{reg}_{A(k)|\log \ep|}\right)\\
&\geq (u_+-\ep^k)(1-\ep^k)\geq u_+-2\ep^k
\end{align*}
Therefore,
$P^\ep_z(\VV(\BB(t))=1)\geq u_+-3\ep^k.$
\qed

\subsubsection{Slope of the interface}
To prepare the proof of Theorem \ref{main}, i.e., the extension of Theorem \ref{thm2.5} to higher dimensions, we state the following result on the ``slope" of the interface.
\begin{prop}\label{prop2.11}
Suppose $x\geq 0$ and $\eta>0$. Then for any time-labelled regular tree $\TT$ with $N_0+1$ offspring and any time $t$,
$$P^t_x(\TT)-P^t_{x-\eta}(\TT)\geq P^t_{x+\eta}(\TT)-P^t_x(\TT).$$
\end{prop}

\begin{proof}
The proof is essentially the same as that of Proposition 2.11 in \cite{Etheridge}. 
We prove the result by induction on the number of branching events in $\TT$. We begin by noting that for a time-labelled tree $\TT_0$ with a root and a single leaf, we easily get 
$$
P^t_x(\TT_0)-P^t_{x-\eta}(\TT_0)=\int_{x-\eta}^x \phi_{0,t}(u)\; du\geq \int_{x}^{x+\eta} \phi_{0,t}(u) \;du= P^t_{x+\eta}(\TT_0)-P^t_x(\TT_0)
$$
where $\phi_{\mu,\sigma^2}$ is the density function of a $N(\mu,\sigma^2)$ random variable. To do the induction step let $\tau$ be the first branching time and let $\TT_1,\dots,\TT_{N_0+1}$ be the trees of the offspring of that branching. We have
\begin{align*}
&(P^t_x(\TT)-P^t_{x-\eta}(\TT))-( P^t_{x+\eta}(\TT)-P^t_x(\TT))\\
=&\left(E_x[g(P^{t-\tau}_{B_\tau}(\TT*))]-E_{x-\eta}[g(P^{t-\tau}_{B_\tau}(\TT*))]\right)-\left(E_{x+\eta}[g(P^{t-\tau}_{B_\tau}(\TT*))]-E_{x}[g(P^{t-\tau}_{B_\tau}(\TT*))]\right).
\end{align*}
 If we let $\rho(x) = g(P^{t-\tau}_{x}(\TT*))$ then the above is
\begin{align*}
&- \int_{-\infty}^\infty g(\rho(y+\eta)) - 2 g(\rho(y)) + g(\rho(y-\eta)) \phi_{x,\tau}(y)\; dy\\
=-&\int_0^\infty g(\rho(y+\eta)) - 2 g(\rho(y)) + g(\rho(y-\eta))
(\phi_{x,\tau}(y)-\phi_{x,\tau}(-y))\; dy
\end{align*}
Since $x\geq 0$, we have $\phi_{x,\tau}(y)-\phi_{x,\tau}(-y)\geq 0$ for $y\ge 0$ so it is enough to show \eqref{G5}, i.e.,
$$
g(\rho(y+\eta)) - 2 g(\rho(y)) + g(\rho(y-\eta)) \le 0.
$$
By the induction assumption $\rho(y) - \rho(y-\eta) \ge \rho(y+\eta) - \rho(y)\equiv h$. Let $p= \rho(y)$.
$$
g(\rho(y+\eta)) -  g(\rho(y)) =g(p+h) - g(p) \le g(p) - g(p-h) \le g(\rho(y)) - g(\rho(y-\eta)) 
$$
by monotonicity of $g$, which completes the proof.
\end{proof}

Exploiting the ``concavity" in Proposition \ref{prop2.11} gives a lower bound on the ``slope" of the interface.

\begin{corollary}\label{cor2.12}
Fix any $T^*\in (0,\infty)$. Suppose that for some $t\in [0, T^*]$ and $z\in \RR$,
\beq\label{2.27}
\big|P^{\ep}_z[\VV(\BB(t))=1]-u_0\big|\leq (u_+-u_0) - \delta_0,
\eeq
Take $\ep_1(1)$ and $c_1(1)$ from Theorem \ref{thm2.5} and $\ep<\min(\ep_1(1),\delta_0/2)$, and let $w\in\RR$ with $|z-w|\leq c_1(1)\ep|\log \ep|$. Then
$$
|P^{\ep}_z[\VV(\BB(t))=1]-P^{\ep}_w[\VV(\BB(t))=1]|
\geq \frac{\delta_0|z-w|}{4c_1(1)\ep|\log \ep|}.
$$
\end{corollary}

\begin{proof}
Consider first the case $0\leq z<w$.  By Theorem \ref{thm2.5} and \eqref{2.27} we have for small $\ep$
$$
P^\ep_{c_1(1)\ep|\log \ep|}[\VV(\BB(t))=1]-P^\ep_{z}[\VV(\BB(t))=1]
\geq \frac{\delta_0}{2}.
$$
Write $\eta=w-z$. 
Proposition \ref{prop2.11} implies that $P^t_{(j+1)\eta+z}-P^t_{j\eta+z}\leq P^t_w-P^t_z$ for $j\in\mathbb{N}$. Let $n_0=\lceil \eta^{-1}(c_1(1)\ep|\log \ep|-z)\rceil$. Then 
$$
P^t_{c_1(1)\ep|\log \ep|}-P^t_z\leq \sum_{j=0}^{n_0-1} P_{(j+1)\eta+z}^t-P_{j\eta+z}^t\leq n_0( P^t_w-P^t_z).
$$
That is,
$$
P^t_w-P^t_z\geq \frac{P^t_{c_1(1)\ep|\log \ep|}-P^t_z}{n_0}
\geq \frac{\delta_0|z-w|}{2(c_1(1)\ep|\log \ep|+|z-w|) }
\geq \frac{\delta_0|z-w|}{4c_1(1)\ep|\log \ep|}.
$$
\end{proof}

\subsection{BBM in higher dimensions}

\subsubsection{Properties of motion by mean curvature}

A key fact in the proof in Etheridge et al \cite{Etheridge} is a coupling between a one dimensional Brownian motion $B_s$ and $d(W_s, t-s)$, the signed distance from a d-dimensional Brownian motion $W_s$ to the interface $\Gamma_{t-s}$. To prepare for the coupling we will state 
some regularity properties of the mean curvature flow, which are given in Section 2.3 of \cite{Etheridge} and are derived based on assumptions (C1)-(C3). Recall that $d(x,t)$ is the signed distance from $x$ to the mean curvature flow $\Gamma_t$.  
\begin{enumerate}
\item There exists $\kappa_0>0$ such that for all $t\in [0,T^*]$ and $x\in \{t: |d(y,t)|\leq \kappa_0\}$ we have 
\beq
|\nabla d(x,t)|=1.
\eeq
Moreover, $d$ is a $C^{\alpha,\alpha/2}$ function in $\{(x,t): |d(x,t)|\leq \kappa_0, t\leq T^*\}$, where $\alpha>3$ as in (C1).
\item Viewing $\textbf{n}=\nabla d$ as the positive normal direction, for $x\in \Gamma_t$, the normal velocity of $\Gamma_t$ at $x$ is $-\partial_t d(x,t)$, and the curvature of $\Gamma_t$ at $x$ is $-\Delta d(x,t)$.
\item There exists $\kappa_0>0$ such that for all $t\in[0,T^*]$ and $x$ such that $|d(x,t)|\leq \kappa_0$,
\beq\label{2.33}
\bigg|\nabla \left(\partial_t d(x,t)-\Delta d(x,t)\right)\bigg|\leq \kappa_0.
\eeq
\item There exists $v_0,V_0>0$ such that for all $t\in [T^*-v_0]$ and all $s\in[t,t+v_0]$,
\beq\label{2.34}
|d(x,t)-d(x,s)|\leq V_0(s-t).
\eeq
\end{enumerate}

We state Proposition 2.13 in \cite{Etheridge}:
\begin{prop}\label{P2.13}
Let $(W_s)_{s\geq 0}$ denote a $d$-dimensional Brownian motion started at $x\in \RR^d$. Suppose that $t\leq T^*$, $\beta \leq \kappa_0$ and let 
$$T_\beta=\inf (\{ s\in[0,t): |d(W_s,t-s)|\geq \beta\} \cup \{t\}).$$
Then we can couple $(W_s)_{s\geq 0}$ with a one-dimensional Brownian motion $(B_s)_{s\geq 0}$ started from $z=d(x,t)$ in such a way that for $s\leq T_\beta$,
$$B_s-\kappa_0\beta s\leq d(W_s,t-s)\leq B_s+\kappa_0\beta s.$$
\end{prop} 

\noindent

By Lemma \ref{independence} we can establish the results for $\hat{\XX^\ep}$, which will also hold for $\XX^\ep$ with high probability. Let $W_t$ denote a Brownian motion in $\RR^d$ while $\hat{X}^\ep_t$ denote a random walk on $\eta\ZZ^d$ with jump rate $\eta^{-2}/2$ to each neighboring site.

\subsubsection{Generation of the interface}\label{genint}
The following proposition is very similar to Proposition 2.15 in \cite{Etheridge}. The major difference is that we work with the rescaled dual process $\XX^\ep_t$ and its comparison process $\hat \XX^\ep_t$ instead of the branching Brownian motion $\WW_t$ in $\RR^d$. 
\begin{prop}\label{215}
Let $k\in \mathbb{N}$ and $\sigma_1(k)$ be defined as in Lemma \ref{upden}. Then there exist $\ep_d(k)$, $a_d(k)$, $b_d(k)>0$ such that for all $\ep\in(0,\ep_d)$, if we set 
\begin{align*}
\delta_d(k,\ep)&:=\max\{a_d(k),\sigma_1(k)\}\ep^2|\log \ep| \\
\delta'_d(k,\ep)&:=(\max\{a_d(k),\sigma_1(k)\}+k+1)\ep^2|\log \ep|,
\end{align*}
then for $t\in[\delta_d,\delta'_d]$,
\begin{enumerate}
\item for $x$ such that $d(x,t)\geq b_d\ep|\log \ep|$, we have $P^\ep_x( \VV_p(\hat{\XX^\ep}(t))=1)\geq u_+- \ep^k$;
\item for $x$ such that $d(x,t)\leq -b_d\ep|\log \ep|$, we have $P^\ep_x( \VV_p(\hat{\XX^\ep}(t))=1)\leq u_-+\ep^k$.
\end{enumerate}
\end{prop}
\begin{proof}
For fixed $k\in \mathbb{N}$ and $A(k)$ specified as in Lemma \ref{lemma2.8}, it follows from Lemma \ref{lem2.9} that there exists $a_d(k), \ep_d(k)>0$ such that for all $\ep\in(0,\ep_d)$ and $t\geq a_d\ep^2|\log \ep|$.
$$
P^\ep [\TT(\hat{\XX}^\ep(t))\supseteq \TT^{reg}_{A(k)|\log \ep|}]\geq 1-\ep^k.
$$
It follows from the same argument as in Lemma \ref{lem2.10} that for $t\in[\delta_d,\delta'_d]$ there exists $b'_d(k),\ep_d(k)$ such that for all $\ep\in(0,\ep_d)$,
$$
P^\ep_x[\exists i\in N(t): |W_i(t)-x|\geq b'_d(k)\ep|\log \ep|] \leq \ep^k.
$$
By (2.34) in \cite{Etheridge} there exists $v_0,V_0>0$ such that for $t\leq v_0$, and any $x\in \RR^d$ we have $|d(x,0)-d(x,t)|\leq V_0t$. We can choose $\ep_d$ sufficiently small so that $\delta'_d\leq v_0$. Thus if $d(x,t)\geq 2b'_d\ep|\log \ep|$ and $|W_i(t)-x|\leq b'_d\ep|\log \ep|$ then
\begin{align*}
d(W_i(t),0)&\geq d(x,t)-|d(x,t)-d(W_i(t),t)|-|d(W_i(t),t)-d(W_i(t),0)|\\
&\geq 2b'_d\ep|\log \ep|-b'_d\ep|\log \ep|-V_0\delta'_d\geq \frac{2}{3}b'_d\ep|\log \ep|.
\end{align*}
It follows from Lemma \ref{err} that 
$$P( |W_i(t)-\hat{X}^\ep_i(t)|\geq \ep \text{ for some }t\leq \delta'_d)\leq \ep^{2k}.$$
The triangle inequality then implies that with probability at least $1-\ep^{2k}$
\begin{align*}
d(\hat{X}^\ep_i(t),0)&\geq d(W_i(t),0)-|\hat{X}^\ep_i(t)-W_i|\geq \frac{2}{3}b'_d\ep|\log \ep|-\ep\geq \frac{1}{2}b'_d\ep|\log \ep|.
\end{align*}

Applying (C2) and (C3),
$$p(\hat{X}^\ep_i(t))\geq u_0+\gamma(\frac{1}{2}b'_d\ep|\log \ep|\wedge r)\geq u_0+\ep.$$
For $x$ such that $d(x,t)\geq 2b'_d\ep|\log \ep|$ and $t\in[\delta_d,\delta'_d]$ it follows exactly from the proof of Theorem \ref{thm2.5} that
$$P^\ep_x[\VV_p(\hat{\XX}^\ep(t))=1]\geq u_+-3\ep^k.$$
Taking $b_d=2b'_d$ completes the proof.
\end{proof}

\subsubsection{Propagation of the interface}

In the Section \ref{genint} we established the existence of an interface develops for a short time interval $[\delta_d,\delta'_d]$. In this section we will show that the interface continue to exist for much longer.
The key to proving Theorem \ref{main} is the following proposition, which is an analogue of Proposition 2.17 in \cite{Etheridge}. To make things easier to write we define $\gamma(t)=K_1e^{K_2t}$ and introduce
$$
z^{\pm}_0 = d(x,t)\pm K_1e^{K_2t} \ep|\log\ep|
$$
which are two points in $\RR$. They depend on $x$ and $t$ but we do not record the dependence in notation.

\begin{prop}\label{2.16}
Let $l\in \mathbb{N}$ with $l\geq 4$. Define $\delta_d(l)$ as in Proposition \ref{215} and $C_1$ as in Lemma \ref{217}.There exists $K_1(l), K_2(l)>0$ and $\ep_d(l,K_1,K_2)>0$ so that for all $\ep\in(0,\ep_d)$ and $t\in[\delta_d(l), T^*]$ we have
\beq\label{2.39}
\sup_{x\in\RR^d} \left( P^\ep_x[\VV_p(\hat{\XX}^\ep(t))=1]
-P^\ep_{z^+_0}[\VV(\BB(t))=1]\right)\leq C_1\ep^{l}
\eeq
\beq\label{2.40}
\sup_{x\in\RR^d} \left( P^\ep_x[\VV_p(\hat{\XX}^\ep(t))=0]
-P^\ep_{z^-_0}[\VV(\BB(t))=0]\right)\leq C_1\ep^{l}
\eeq
\end{prop}

The key ingredient for proving Proposition \ref{2.16} is the following lemma, which is an analogue of Lemma 2.18 in \cite{Etheridge}. The idea of the proof remains the same but the coefficients are slightly different due to the differences in the $g$'s. Let $B_t\in \RR$ be a one-dimensional Brownian motion that can be thought of as a single lineage in the branching Brownian motion $\BB(t)$.
\begin{align*}
z^\pm_1 & = d(\hat{X}^\ep_s,t-s)\pm \gamma(t-s)\ep|\log \ep|\\
z^\pm_2 & = B_s\pm \gamma(t)\ep|\log \ep|
\end{align*}
\begin{lemma}\label{217}
Let $l\in \mathbb{N}$ with $l\geq 4$ and $\sigma_1(l)$ be as in Lemma \ref{upden}. Let $\delta_0$ and $c_0$ be chosen as in (G4). Choose $C_1$ sufficiently large so that $C_1>\max\{ 2(1-c_0)/c_0, 3/(2c_0)\}.$
Let $C_2=\max_{0\leq p\leq 1}C_1|g'(p)|$. Let $K_1>0$. There exists $K_2=K_2(K_1,l)>0$ and $\ep_d(l,K_1,K_2)>0$ such that for all $\ep\in(0,\ep_d), x\in\RR^d$, $s\in [0,(l+1)\ep^2|\log \ep|]$ and $t\in [s,T^*]$,
\begin{align}\label{2.42}
\nonumber E_x\bigl[g(P^{\ep}_{z^+_1}[\VV(\BB(t-s))&=1]+C_1\ep^{l})\bigr]\\
&\leq (1-c_0/3)C_1\ep^l+E_{d(x,t)}\bigl[g(P^{\ep}_{z^+_2}[\VV(\BB(t-s))=1])\bigr]
+C_2\ep^l1_{s\leq \ep^4}\\
\nonumber E_x\bigl[g(P^{\ep}_{z^-_1}[\VV(\BB(t-s))&=0]+C_1\ep^{l})\bigr]\\
&\leq (1-c_0/3)C_1\ep^l+ E_{d(x,t)}\bigl[g(P^{\ep}_{z^-_2}[\VV(\BB(t-s))=0])\bigr]
+C_2\ep^l1_{s\leq \ep^4}
\end{align}
\end{lemma}

To keep our approach parallel to the one in \cite{Etheridge} we defer the proof of Lemma \ref{217} to the next subsection. The only property of $g$ that is used in the proof below is its monotonicity.

\mn
\textit{Proof of Proposition \ref{2.16}.}  We begin by proving (\ref{2.39}) for $t\in[\delta_d,\delta_d']$. Take $K_1=b_d(l)+c_1(l)$ where $b_d(l)$ is as defined in Proposition \ref{215} and $c_1$ is as defined in Theorem \ref{thm2.5}. Let $K_2=K_2(K_1,l)$, as defined in Lemma \ref{217}. If $d(x,t)\leq -b_d(l)\ep|\log \ep|$, then by Proposition \ref{215}, $P^\ep_x[\VV_p(\hat{\XX}^\ep(t))=1]\leq \ep^{l}$. Then (\ref{2.39}) holds.

On the other hand, if $d(x,t)\geq -b_d(l)\ep|\log \ep|$, then $d(x,t)+\gamma(t)\ep|\log \ep|\geq c_1(l)\ep|\log \ep|$, and by Theorem \ref{thm2.5}
$$P^\ep_{d(x,t)+\gamma(t)\ep|\log \ep|}[\VV(\BB(t))=1]\geq u_+-\ep^{l}.$$
By definition of $\delta_d$ in Proposition \ref{215}, $t\geq \sigma_1(l)\ep|\log \ep|$. It follows from the same argument as in Lemma \ref{upden} that 
$$P^\ep_x[\VV_p(\hat{\XX}^\ep(t))=1]\leq u_++\ep^{l}.$$
Therefore when $\ep$ is sufficiently small (\ref{2.39}) holds.

We follow the proof in \cite{Etheridge} and assume that there exists $t\in[\delta'_d,T^*]$ such that for some $x\in\RR^d$ \eqref{2.39} does not hold, i.e.,
$$ P^\ep_x[\VV_p(\hat{\XX}^\ep(t))=1]-P^\ep_{d(x,t)+\gamma(t)\ep|\log \ep|}[\VV(\BB(t))=1]> C_1\ep^{l}.$$
Let $T'$ be the infimum of the set of such $t$ and choose
$$T\in[ T', \min( T'+\ep^{l+3}, T^*)]$$
which is in the set of such $t$. Hence there exists some $x\in\RR^d$ such that 
\beq\label{2.46}
P^\ep_x[\VV_p(\hat{\XX}^\ep(T))=1]-P^\ep_{d(x,T)+\gamma(T)\ep|\log \ep|}[\VV(\BB(T))=1]>C_1\ep^{l}.
\eeq
Our goal is to contradict \eqref{2.46} by showing that 
\beq\label{2.47}
P^\ep_x[\VV_p(\hat{\XX}^\ep(T))=1]\leq P^\ep_{d(x,T)+\gamma(T) \ep|\log \ep|}[\VV(\BB(T))=1]+(1-c_0/4)C_1\ep^{l}.
\eeq
We write $S$ for the time of the first branching event in $\hat{\XX}^\ep(T)$ and $\hat{X}^\ep(S)$ for the position of the initial particle at that time. By the strong Markov property
\begin{align}\label{2.48}
\nonumber P^\ep_x[\VV_p(\hat{\XX}^\ep(T))=1]&\leq E^\ep_x[g(P^\ep_{\hat{X}^\ep(S)}[\VV_p(\hat{\XX}^\ep(T-S))=1]1_{S\leq T-\delta_d}]\\
&+E^\ep_x[P^\ep_{\hat{X}^\ep(T-\delta_d)}[\VV_p(\hat{\XX}^\ep(\delta_d))=1]1_{S\geq T-\delta_d}]
\end{align}
Let $c^*$ be a constant such that $c^*\ep^{-2}$ is the reaction rate for the process that we consider. For sexual reproduction model with fast stirring, $c^*=(1+\beta)$ as defined in \eqref{cstar}. For voter model perturbations, $c^*=1$. Without loss of generality we can assume $c^*\geq 1$ since otherwise we can rescale time to obtain $c^*\geq 1$. Since $S =Exponential(c^*\ep^{-2})$ and $T-\delta_d\geq \delta'_d-\delta_d=(l+1)\ep^2|\log \ep|$, we have 
$$
E^\ep_x\left[P^\ep_{\hat{X}^\ep_{T-\delta_d}}[\VV_p(\hat{\XX}^\ep(\delta_d))=1]1_{\{S\geq T-\delta_d\}}\right]\leq P[S\geq (l+1)\ep^2|\log \ep|]\leq \ep^{c^*(l+1)}\leq \ep^{l+1}.
$$
To bound the first term in \eqref{2.48}, partition on the event $\{S\leq \ep^{l+3}\}$,
\begin{align}\label{2.50}
\nonumber &E^\ep_x[g(P^\ep_{\hat{X}^\ep_S}[\VV_p(\hat{\XX}^\ep(T-S))=1]
1_{\{S\leq T-\delta_d\}}]\\
\nonumber \leq &P[S\leq \ep^{l+3}]+
E^\ep_x[g(P^\ep_{\hat{X}^\ep_S}[\VV_p(\hat{\XX}^\ep(T-S))=1]
1_{\{\ep^{l+3}\leq S\leq T-\delta_d\}}]\\
\leq &\ep^{l+1}+E^\ep_x[ g( P^\ep_{d(\hat{X}^\ep_S,T-S)+\gamma(T-S)\ep|\log \ep|}[\VV(\BB(T-S))=1]+C_1\ep^{l})1_{\{S\leq T-\delta_d\}}].
\end{align}
The last line follows from the minimality of $T'$ and the fact that $T-S\leq T'$ on the event $\{ S\geq \ep^{l+3}\}$.
\begin{align*}
&E^\ep_x[ g( P^\ep_{d(\hat{X}^\ep_S,T-S)+\gamma(T-S)\ep|\log \ep|}
[\VV(\BB(T-S))=1]+C_1\ep^{l})1_{S\leq T-\delta_d}] \\
\leq & \int_0^{(l+1)\ep^2|\log \ep|} c^*\ep^{-2} e^{-c^*\ep^{-2s}} E_x[g(P^\ep_{d(\hat{X}^\ep_s,T-s)+K_1e^{K_2(T-s)}\ep|\log \ep|}[\VV(\BB(T-s))=1]+C_1\ep^{l})]\;ds \\
\nonumber&+P[S\geq (l+1)\ep^2|\log\ep|].
\end{align*}
Using  Lemma \ref{217} we get
\begin{align*}
\nonumber\leq & (1-c_0/3)C_1\ep^l+ \int_0^{(l+1)\ep^2|\log \ep|} c^*\ep^{-2} 
e^{-c^*\ep^{-2}s} E_{d(x,t)}[g(P^{\ep}_{B_s+\gamma(t)\ep|\log \ep|}[\VV(\BB(t-s))=1])]\;ds\\
\nonumber&+C_2\ep^l P[S\leq \ep^4]+\ep^{l+1}.
\end{align*}
Let $S'$ denotes the first branching time in $(\BB(s))_{s\geq 0}$ and $B_{S'}$ the position of the ancestor at that time. Noting that $S'$ has the same distribution as $S$ we have
\beq\label{2.51}
\leq (1-c_0/3)C_1\ep^l+2\ep^{l+1}+E^\ep_{d(x,t)}[g(P^\ep_{B_{S'}+K_1e^{K_2T}\ep|\log \ep|}[\VV(\BB(T-S'))=1]1_{S'\leq T-\delta_d'}]].
\eeq
 Combining \eqref{2.48}, \eqref{2.50} and \eqref{2.51},
\begin{align*}
P^\ep_x[\VV_p(\hat{X}^\ep(T))=1]&\leq 4\ep^{l+1}+(1-c_0/3)C_1\ep^l+E^\ep_{d(x,t)}[g(P^\ep_{B_{S'}+K_1e^{K_2T}\ep|\log \ep|}[\VV(\BB(T-S'))=1]]\\
&\leq (1-c_0/4)C_1\ep^l+P^\ep_{d(x,T)+K_1e^{K_2T}\ep|\log \ep|}[\VV(\BB(T))=1],
\end{align*}
which proves \eqref{2.47} and hence we have proved \eqref{2.39} by an argument of contradiction. The proof of \eqref{2.40} is similar.
\qed

Before giving the proof of Lemma \ref{217} we prove the main result.

\begin{theorem}\label{main}
Let $u^\ep(t,x)=P(\xi^\ep_t(x)=1)$ with $u^\ep(0,x)=p(x)$. Let $T^*\in(0,\mathscr{T})$ and $k\in \mathbb{N}$ be fixed. Choose $\sigma_1(k)$  as in Lemma \ref{upden}. There exist $\ep_d(k)>0$ and $a_d(k),c_d(k)\in (0,\infty)$ such that for all $\ep\in(0,\ep_d)$ and $t$ satisfying $\max\{a_d, \sigma_1\}\ep^2|\log \ep|\leq t\leq T^*$,

\begin{enumerate}
\item for $x$ such that $d(x,t) \geq c_d(k)\ep|\log \ep|$, we have $u^{\ep}(t,x)\geq u_+-\ep^k$,
\item for $x$ such that $d(x,t)\leq -c_d(k)\ep|\log \ep|$, we have $u^{\ep}(t,x)\leq u_-+\ep^k$.
\end{enumerate}
\end{theorem}

\begin{proof}
We first prove the result for $\hat{\XX}^\ep(t)$. We choose $c_d(k)=c_1(k)+K_1e^{K_2T^*}$. Thus for $t\in[\delta_d,T^*]$ and $x\in\RR^d$ such that $d(x,t)\leq -c_d(k)\ep|\log \ep|$ we have 
$$d(x,t)+K_1e^{K_2T^*}\leq -c_1(k)\ep|\log \ep|.$$
It follows from Proposition \ref{2.16} and Theorem \ref{thm2.5} that $P^\ep_x[\VV_p(\hat{\XX}^\ep(t))=1]\leq u_-+(C_1+1)\ep^k$. Similarly, if $d(x,t)\geq c_d(k)\ep|\log \ep|$ then $d(x,t)-K_1e^{K_2T^*}\geq c_1(k)\ep|\log \ep|$. Hence 
$$
P^\ep_x[\VV_p(\hat{\XX}^\ep(t))=0]\leq P^\ep_{d(x,t)-\gamma(t)\ep|\log \ep|}[\VV(\BB(t))=0]+C_1\ep^k\leq 1-u_++(1+C_1)\ep^k.
$$
It remains to show $u^\ep(t,x)$ is close to $P^\ep_x[\VV_p(\hat{\XX}^\ep(t))=1]$. Let $G=\{\XX^\ep(t)=\hat{\XX}^\ep(t) \text{ for } t\leq T^*\}$. Lemma \ref{independence} implies that $P(G)\geq 1-\ep^k$.

Then 
\begin{align*}
u^\ep(t,x)&=P^\ep_x[\VV_p(\XX(t))=1]\\
&=P^\ep_x[\{\VV_p(\hat{\XX}^\ep(t))=1\}\cap G ]+P^\ep_x[\{\VV_p(\XX(t))=1\}\cap G^c ]\\
&\leq P^\ep_x[\VV_p(\hat{\XX}^\ep(t))=1]+\ep^k
\end{align*}
On the other hand, 
$$u^\ep(t,x)\geq  P^\ep_x[\{\VV_p(\hat{\XX}^\ep(t))=1\}\cap G ]\geq P^\ep_x[\VV_p(\hat{\XX}^\ep(t))=1]-P[G^c]\geq P^\ep_x[\VV_p(\hat{\XX}^\ep(t))=1]-\ep^k .$$
Therefore, $| u^\ep(t,x)-P^\ep_x[\VV_p(\hat{\XX}^\ep(t))=1]|\leq \ep^k.$
\end{proof}

\subsubsection{Proof of Lemma \ref{217}}

\begin{proof}
We continue to write $\gamma(t)=K_1e^{K_2t}$. Define a good event by 
$$
G=\{ |d(W_s,t-s)-d(\hat{X}^\ep_s,t-s)| \leq \ep \text{ for }s\in [0, (l+1)\ep^2|\log \ep|]\}.$$
The triangle inequality implies $d(W_s,t-s) \leq d(\hat{X}^\ep_s,t-s)+|\hat{X}^\ep_s-W_s|$. There is a similar result with $W$ and $X$ interchanged so
\beq
|d(W_s,t-s)-d(\hat{X}^\ep_s,t-s)|\leq |\hat{X}^\ep_s-W_s|
\label{dWX}
\eeq
Lemma  \ref{err} implies that for sufficiently small $\ep$
\beq
P(G)\geq 1-\ep^{2l}.
\label{PGbd}
\eeq

We choose $\kappa_0$ as in \eqref{2.33} and $c_1(k)$ from Theorem \ref{thm2.5}. Let \beq
R=2c_1(l)+4(l+1)d+1
\label{defR}
\eeq 
and fix $K_2$ such that 
\beq\label{2.53}
(K_1+1)(K_2-\kappa_0)-\kappa_0R=c_1(1).
\eeq
Let $s\in [0, (l+1)\ep^2|\log \ep|]$ and 
$$
A_x=\left\{ \sup_{u\in[0,s]} |W_u-x|\leq 2(l+1)d\ep|\log \ep| \right\}.
$$
Using the reflection principle
\begin{align}
P(A_x^c) & \leq 2d P_0\left( \sup_{u \in [0,s]} B_u > 2(\ell+1) \ep|\log\ep| \right)
\nonumber\\
&\le 4d P_0 ( B_s > 2(\ell+1) \ep|\log\ep| ) \le 4d \ep^{l+1}
\label{PAxbd}
\end{align} 
where we have used the tail bound  the tail bound 
$$
P(B_s \ge x\sqrt{s}) \le \exp(-x^2/4)
$$
with $s=(l+1)\ep^2|\log \ep|$ and  $x = 2\sqrt{(\ell+1)|\log \ep|}$.

\mn
Recall that in Lemma \ref{217} $s\in[0,(\ell+1)\ep^2|\log\ep|]$ is fixed and $t\in[s,T^*]$.
We consider three cases:
\begin{enumerate}
\item $d(x,t)\leq -\left(2c_1(l)+2(l+1)d+\gamma(t-s)\right) \ep|\log \ep|$,
\item $d(x,t)\geq \left(2c_1(l)+2(l+1)d+\gamma(t-s)\right) \ep|\log \ep|$,
\item $|d(x,t)|\leq \left(2c_1(l)+2(l+1)d+\gamma(t-s)\right) \ep|\log \ep|$.
\end{enumerate}
The first two are easy since $x$ is far from the interface so the probabilities of interest are either close to $u_+$ or close to $u_-$.

\mn
\textbf{Case 1}: By \eqref{2.34} there exists $v_0,V_0>0$ such that if $s\leq v_0$ and $x\in \RR^d$ then 
\beq
|d(x,t)-d(x,t-s)|\leq V_0s.
\label{dvst}
\eeq
We take $\ep_d$ sufficiently small in Lemma \ref{217} so that $(l+1)\ep^2|\log \ep|\leq v_0$ for all $\ep\in (0,\ep_d)$. Rearranging the definition of Case 1 and adding $d(W_s,t-s)$ to both sides
$$
 d(W_s,t-s)+\gamma(t-s)\ep|\log \ep|\leq 
-\left(2c_1(l)+2(l+1)d\right) \ep|\log \ep|+ d(W_s,t-s)-d(x,t)
$$
The triangle inequality implies $d(x,t-s)+|W_s-x| \ge d(W_s,t-s)$ so 
$$
 d(W_s,t-s)+\gamma(t-s)\ep|\log \ep|\leq -\left(2c_1(l)+2(l+1)d\right) \ep|\log \ep|+|W_s-x|+|d(x,t)-d(x,t-s)|.
$$
Using \eqref{dvst} with $s\leq (l+1)\ep^2|\log \ep|$we see that on $A_x$
$$
 d(W_s,t-s)+\gamma(t-s)\ep|\log \ep|\leq -2c_1(l)\ep|\log \ep|+V_0(l+1)\ep^2|\log \ep|. 
$$
On event $G\cap A_x$ when $\ep$ is sufficiently small, 
\begin{align*}
z^+_1 = d(\hat{X}^\ep_s,t-s)+\gamma(t-s)\ep|\log \ep| 
&\leq d(W_s,t-s)+\ep+\gamma(t-s)\ep|\log \ep|\\
&\leq -c_1(l)\ep|\log \ep|.
\end{align*}
Hence it follows from Theorem \ref{thm2.5} that
\begin{align*}
E_x&[g(P^{\ep}_{z^+_1}
[\VV(\BB(t-s))=1]+C_1\ep^{l})]
\leq E_x[g(u_-+(1+C_1)\ep^l)]+P_x[A^c_x]+P[G^c].
\end{align*}
Using (G4), \eqref{PAxbd}, and \eqref{PGbd} the above is
$$
\leq u_-+(1-c_0)\cdot (1+C_1) \ep^l +4d\ep^{l+1}+\ep^{2l} \leq u_-+(1-c_0/3)C_1\ep^l
$$
when $\ep$ is sufficiently small. As $u_-$ is a fixed point of $g$ and we start with initial condition $p_0(x)=u_+\cdot1_{\{x\geq 0\}}+u_-\cdot1_{\{x< 0\}}$ for the one dimensional BBM, the second term on the right hand side of \eqref{2.42} satisfies
$$E_{d(x,t)}\bigl[g(P^{\ep}_{z^+_2}[\VV(\BB(t-s))=1])\bigr]\geq u_-.$$
The third term on the right hand side of \eqref{2.42} is non-negative so the result follows.

\mn
\textbf{Case 2}: In this case $d(x,t)\geq (c_1(l)+2(l+1))\ep|\log \ep|$. Repeating the proof of \eqref{PAxbd} gives
\beq
P_{d(x,t)}[B_s\leq c_1(l)\ep|\log \ep|]
\leq P_0[B_s\geq 2(l+1)\ep|\log \ep|]\leq \ep^{l+1}
\label{BMbd}
\eeq
Recall $z^+_2 = B_s+\gamma(t)\ep|\log \ep|$. Using Theorem \ref{thm2.5} and \eqref{BMbd} and \eqref{G4} it follows that 
\begin{align*}
E_{d(x,t)}[&g(P^{\ep}_{z^+_2}[\VV(\BB(t-s))=1])]\\
&\geq E_{d(x,t)}[g(P^{\ep}_{z^+_2}[\VV(\BB(t-s))=1])
1_{\{B_s\geq c_1(l)\ep|\log \ep|\}}]\\
&\geq  g(u_+-\ep^{l})-\ep^{l+1}\geq u_+-(1-c_0)\ep^{l}-\ep^{l+1}
\geq u_+ -\ep^{l}
\end{align*}
when $\ep$ is small. Therefore, the right hand side of \eqref{2.42} for small $\ep$ is at least 
$$
(1-c_0/3)C_1\ep^l+u_+-\ep^{l}.
$$
Since the initial condition is $p_0(x)=u_+\cdot1_{\{x\geq 0\}}+u_-\cdot1_{\{x< 0\}}$, by the monotonicity of $g$ it is easy to see that for any $x\in \RR$ and $t\geq 0$,
$$
P^\ep_{x}[\VV(\BB(t))=1]\leq u_+.
$$
Hence using (G4) the left hand side of \eqref{2.42} is 
\begin{align*}
E_x[g(P^{\ep}_{z^+_1}[\VV(\BB(t-s))=1]+C_1\ep^{l})]
&\leq E_x[g(u_++C_1\ep^l)]\\
&\leq u_++(1-c_0) \cdot C_1 \ep^l \leq u_++ ((1-c_0/3)C_1-1)\ep^l,
\end{align*}
where the last line follows from the choice of $C_1$. So \eqref{2.42} holds in this case.

\mn
\textbf{Case 3}: We now turn to the case with 
$$
|d(x,t)|\leq \left(2c_1(l)+2(l+1)d+\gamma(t-s)\right) \ep|\log \ep|. 
$$
Using \eqref{dvst} we see that on the event $A_x$,  we have for $u\in[0,s]$
\begin{align*}
|d(W_u,t-u)|&\leq |W_u-x|+|d(x,t)|+|d(x,t)-d(x,t-u)|\\
&\leq \left(2c_1(l)+4(l+1)d+\gamma(t-s)\right) \ep|\log \ep|
+V_0(l+1)\ep^2|\log \ep|\\
&\leq (R+\gamma(t-s))\ep|\log \ep|,
\end{align*}
where $R=2c_1(l)+4(l+1)d+1$, see \eqref{defR}. Applying Proposition \ref{P2.13} with 
$$
\beta=(R+\gamma(t-s))\ep|\log \ep|
$$
 shows we can couple $(W_u)_{u\geq 0}$ with $(B_u)_{u\geq 0}$ (which starts from $d(x,t)$) in such a way that for 
$u\leq T_\beta = \inf\{ s \in [0,t) : |d(W_s,t-s)| >\beta \} \wedge t$,
$$
d(W_u,t-u)\leq B_u+\kappa_0\beta u.
$$
Note that $ A_x\subseteq\{T_\beta>s\}$. Let $\eta>0$. 
Recall $z^+_1  = d(\hat{X}^\ep_s,t-s) + K_1e^{K_2(t-s)}\ep|\log \ep|$ and let
\begin{align*}
&z_3^+ = d(W_s,t-s)+\ep+\gamma(t-s)\ep|\log \ep|\\
&z^+_4 = B_s+\kappa_0\beta s+\ep+\gamma(t-s)\ep|\log \ep|
\end{align*}
By the coupling between $d(W_t,t-s)$ and $B_s$ we have $z_3^+\leq z_4^+$.
By the convergence of $\hat{X}^\ep_s$ to  $W_s$ proved in Lemma \ref{err} and the monotonicity of $g$
\begin{align}\label{2.61}
\nonumber E_x[&g(P^{\ep}_{z^+_1}[\VV(\BB(t-s))=1]+C_1\ep^{l})]\\
\nonumber &\leq E_x[g(P^{\ep}_{z^+_3}[\VV(\BB(t-s))=1]+C_1\ep^{l})]
+P_x(A_x^c) + P(G^c)\\
&\leq  E_{d(x,t)}[g(P^{\ep}_{z^+_4}[\VV(\BB(t-s))=1]+C_1\ep^l)]+4d\ep^{l+1}+\ep^{2l}.
\end{align}
where in the last step we have used \eqref{PAxbd}. Let 
$$
E=\{ |P^{\ep}_{z^+_4}[\VV(\BB(t-s))=1]-u_0|\leq (u_+-u_0)-\delta_0\}.
$$
where $\delta_0$ is the constant defined before \eqref{G4}.

Consider first when the event $E$ occurs. 
\begin{align}
\nonumber&\gamma(t)\ep|\log \ep|-(\ep+\kappa_0\beta s+\gamma(t-s)\ep|\log \ep|)\\
\nonumber& \geq \gamma(t)\ep|\log \ep|-(\kappa_0\beta s+(K_1+1)e^{K_2(t-s)}\ep|\log \ep|)\\
\nonumber &= \left( (K_1+1)e^{K_2(t-s)}(e^{K_2s}-1-\kappa_0s)-\kappa_0Rs\right)
\ep|\log \ep|\\
&\geq ((K_1+1)(K_2-\kappa_0)-\kappa_0R)s\ep|\log \ep|
=c_1(1)s\ep|\log \ep| \label{4vs2}
\end{align}
where the last line follows from the choice of $K_2$ in \eqref{2.53}. 
Take $\ep_d$ sufficiently small so that $\ep_d<\min(\ep_1(1),\delta_0/2)$. For $\ep\in (0,\ep_d)$ we can apply Corollary \ref{cor2.12} to $z=z^+_4$ and $w=z^+_2$
Using \eqref{4vs2} to conclude $z_2^+-z_4^+ \ge c_1(1)s\ep|\log\ep|$
it follows that on $E$
\beq
P^\ep_{z^+_2}[\VV(\BB(t-s))=1]
-P^\ep_{z^+_4}[\VV(\BB(t-s))=1] \ge \frac{\delta_0s}{4}
\label{fromslope}
\eeq
so we have
$$
g(P^{\ep}_{z^+_4}[\VV(\BB(t-s))=1]+C_1\ep^l) 
\le g(P^{\ep}_{z^+_2}[\VV(\BB(t-s))=1]- \delta_0 s/4+C_1\ep^l)
$$
Recalling $s \le (\ell+1)\ep^2|\log \ep|$ and using the monotonicity of $g$  we can replace $- \delta_0 s/4+C_1\ep^l$ by 0 when $s > 4C_1 \ep^l/\delta_0$. If $\ell \ge 4$ and $s \le 4C_1 \ep^l/\delta_0$ the $s \le \ep^3$ for small $\ep$. Since $g'(p) \le C_2$ 
\begin{align}\label{2.63}
\nonumber g(P^{\ep}_{z^+_4}[\VV(\BB(t-s))=1]+C_1\ep^l) 
 &\le g(P^{\ep}_{z^+_2}[\VV(\BB(t-s))=1]) + \max_{0\leq p\leq 1}|g'(p)|\cdot C_1\ep^l 
1_{s\le \ep^3}\\
&\leq  g(P^{\ep}_{z^+_2}[\VV(\BB(t-s))=1]) +C_2\ep^l 1_{s\le \ep^3}
\end{align}

\eqref{G4} implies that If $p\geq u_+-\delta_0,\delta\geq 0$ then
\beq\label{gerr}
g(p+\delta)\leq g(p)+(1-c_0)\delta.
\eeq
Taking $\ep_d$ sufficiently small so that $C_1\ep^l<\delta_0$ for all $\ep\in(0,\ep_d)$, and using \eqref{gerr} we have on $E^c$ that
\begin{align}
g(P^\ep_{z^+_4}[\VV(\BB(t-s))=1]+C_1\ep^l)
&\leq g(P^\ep_{z^+_4}[\VV(\BB(t-s))=1]) + (1-c_0) \cdot C_1\ep^l
\nonumber\\
&\leq g(P^\ep_{z^+_2}[\VV(\BB(t-s))=1]) + (1-c_0) \cdot C_1\ep^l
\label{2.65}
\end{align}
since $z^+_4 \le z^+_2$.
Using \eqref{2.63} and \eqref{2.65} in \eqref{2.61} 
\begin{align*}
&E_x[g(P^{\ep}_{z^+_1}[\VV(\BB(t-s))=1]+C_1\ep^{l})]\\
&\leq E_{d(x,t)}\left[ g(P^\ep_{z^+_2}[\VV(\BB(t-s))=1])\right]
+ (1-c_0)C_1\ep^l+4d\ep^{l+1}+\ep^{2l} + C_2 \ep^l  1_{s\le \ep^3}\\
&\leq E_{d(x,t)}\left[ g(P^\ep_{z^+_2}[\VV(\BB(t-s))=1])\right]+(1-c_0/3)C_1\ep^l+C_2 \ep^l  1_{s\le \ep^3},
\end{align*}
which completes the proof of Lemma \ref{217} and hence of Proposition \ref{2.16}.

\end{proof}

\clearp

\section{Checking the conditions} \label{sec:check}
Since (G0) is based on an observation on all the particle systems considered, it is satisfied trivially. Recall that (G5) $g$ is strictly increasing on $[0,1]$ holds in all our examples and (G4) is a consequence of (G1), (G2) and (G3). That is, it suffices to check (G1)-(G3).

\subsection{Cubic $g$}
As discussed in Section 1, both the sexual reproduction model with rapid stirring and the Lotka-Volterra systems fall into this category. In this case, according to \eqref{gphi} we must have
$$
g(p) = p - c[(p-u_-)(p-u_0)(p-u_+)]
$$
for some $c>0$. To check (G1) we note that if we let $g_1(p)=(p-u_-)(p-u_0)(p-u_+)$ then 
$g_1(u_+-\delta)=-g_1(u_-+\delta)$. So $g(u_+-\delta)+g(u_-+\delta) =u_+ + u_- = 2u_0$ by (G0).
$$
g'(p) =  1 - c [(p-u_0)(p-u_+)+ (p-u_-)(p-u_+)+(p-u_-)(p-u_0)]
$$
From this we see that 
\begin{align*}
g'(u_+) & = 1 - c (u_+ - u_-)(u_+-u_0) < 1, \\
g'(u_-) & = 1 - c (u_- - u_0)(u_- -u_+) < 1, \\
g'(u_0) & = 1 - c (u_0 - u_-)(u_0-u_+) > 1, 
\end{align*}
which proves (G2). Taking the second derivative we obtain
$$
g''(p) =  - 2c [(p-u_+)+ (p-u_0) +(p-u_-)] =- 6c(p-u_0)
$$
since $u_+ + u_- = 2u_0$. This proves (G3).

\subsection{Nonlinear voter model}

Recall that for the nonlinear voter model we suppose

\mn
(A1) $b_1 >0$ and $3b_1+b_2 <0$;\\
(A2) $0\le a_1 \le a_2 \le 1/2$;\\
(A3) $6b_1+b_2>0$.

\mn
In Region 2 there are two extra roots of $\phi(p)$ denoted by $1-u^*<1/2<u^*$, where 
$$
u^*=1/2+\beta_0 \quad \text{ with }\beta_0=\frac{\sqrt{-(b_1-b_2)(3b_1+b_2)}}{2(b_1-b_2)}.
$$
The roots come from the following calculation:
\begin{align*}
\phi(p)&=b_1p(1-p)^4+b_2p^2(1-p)^3-b_2p^3(1-p)^2-b_1p^4(1-p)\\
&=b_1p(1-p)(1+3p^2-3p-2p^3)+b_2(1-p)^2p^2(1-2p)\\
&=p(1-p)(1-2p)\left( b_1(1-p+p^2)+b_2p(1-p)\right)\\
&=p(1-p)(1-2p)(b_1-b_2)\left(p^2-p+\frac{b_1}{b_1-b_2}\right).
\end{align*}
Solving $p^2-p+b_1/(b_1-b_2)=0$ gives the two extra roots $\frac{1}{2}\pm\beta_0$. 

To check our conditions we note that $g(p) = p + \phi(p)$ where $\phi(p)$ is the reaction term, see \eqref{gphi}. In our notation $u_0=1/2$, $u_-=1-u^*$ and $u_+=u^*$.

\mn
\textbf{Checking (G1):} $\phi(p)$ is antisymmetric about $u_0$ so $\phi(u_+-\delta)=-\phi(u_-+\delta)$ and hence $g(u_+-\delta)+g(u_-+\delta) =u_+ + u_- = 2u_0$, proving (G1). 

\mn
\textbf{Checking (G2):} $u_-, u_+$ are stable fixed points so $\phi'(u_-)<0, \phi'(u_+)<0$. $u_0$ is unstable so $\phi'(u_0)>0$ and (G2) follows. 

\mn
\textbf{Checking (G3):} 
Since $g''(p)=\phi''(p)$ the next step is to calculate $\phi''(p)$ for $p\in(1/2,u^*)$. By symmetry it is easy to see 
\beq\label{phipro}
\phi(0)=\phi(1/2)=\phi(1)=0 \quad \text{ and }\quad \phi(p)=-\phi(1-p).
\eeq
It follows that $\phi''(p)=-\phi''(1-p)$ and $\phi''(1/2)=0$. Since $\phi(p)$ is quintic it has at most three inflection points. To check (G3) it suffices to show $\phi''(u^*)<0$.

Let $\phi_1(p)=p(1-p)(1-2p)$ and  $\phi_2(p)=(b_1-b_2)\left(p^2-p+\frac{b_1}{b_1-b_2}\right)$. Since $\phi(p)=\phi_1(p)\phi_2(p)$ we have 
$$\phi''(p)=\phi_1''(p)\phi_2(p)+\phi_1(p)\phi_2''(p)+2\phi_1'(p)\phi'_2(p).$$
Notice that $\phi_2(u^*)=0$ so our problem simplifies to
\begin{align*}
\phi''(u^*)
&=\phi_1(u^*)\phi_2''(u^*)+2\phi_1'(u^*)\phi'_2(u^*)
\end{align*}
The calculation simplifies if we write $u^*=1/2+\beta_0$, i.e.,
\begin{align*}
\phi''(1/2+\beta_0)
&=\phi_1(1/2+\beta_0)\phi_2''(1/2+\beta_0)+2\phi_1'(1/2+\beta_0)\phi'_2(1/2+\beta_0)\\
&=-2\beta_0\left(\frac{1}{4}-\beta_0^2\right)\cdot 2(b_1-b_2)+2\left(6\beta_0^2-\frac{1}{2}\right)\cdot 2\beta_0(b_1-b_2)\\
&=4\beta_0(b_1-b_2)\left(7\beta_0^2-\frac{3}{4}\right)=-4\beta_0(6b_1+b_2)<0,
\end{align*} 
hence proving (G3).

\clearp

\end{document}